\documentclass[12pt,a4paper]{article}
\usepackage{amsmath,amsfonts,amssymb,amsthm,amscd,eucal}

\newtheorem{thm}{Theorem}
\newtheorem{lem}[thm]{Lemma}
\newtheorem{prop}[thm]{Proposition}
\newtheorem{defn}[thm]{Definition}
\newtheorem{cor}[thm]{Corollary}
\newtheorem{rem}[thm]{Remark}

\begin{document}

\newfont{\goth}{eufm10 scaled \magstep1}
\def\ga{\mbox{\goth a}}
\def\gp{\mbox{\goth p}}
\def\gc{\mbox{\goth c}}
\def\gg{\mbox{\goth g}}
\def\ge{\mbox{\goth e}}
\def\gh{{\mbox{\goth h}}}
\def\gk{\mbox{\goth k}}
\def\gl{\mbox{\goth l}}
\def\gf{\mbox{\goth f}}
\def\gm{\mbox{\goth m}}
\def\gn{\mbox{\goth n}}
\def\gq{\mbox{\goth q}}
\def\gr{\mbox{\goth r}}
\def\gs{\mbox{\goth s}}
\def\gt{\mbox{\goth t}}
\def\gu{\mbox{\goth u}}
\def\gw{\mbox{\goth w}}
\def\gsl{\mbox{\goth sl}}
\def\gsp{\mbox{\goth sp}}
\def\r{\mbox{\goth r}}
\def\gso{\mbox{\goth so}}
\def\gsu{\mbox{\goth su}}
\newcommand\Id{\mathrm{Id}}
\newcommand\Ad{\mathrm{Ad}}
\newcommand\ad{\mathrm{ad}}

\title{Invariant generalized complex structures on Lie groups}

\author{Dmitri V. Alekseevsky and Liana David}

\maketitle

{\bf Abstract.} We describe a class (called regular) of invariant
generalized complex structures on a real semisimple Lie group $G$.
The problem reduces to the description of admissible pairs $(\gk ,
\omega )$, where $\gk \subset \gg^{\mathbb{C}}$ is an appropriate
regular subalgebra of the complex Lie algebra $\gg^{\mathbb{C}}$
associated to $G$ and $\omega$ is a closed $2$-form on $\gk$, such
that $\mathrm{Im}\left(\omega\vert_{\gk \cap \gg}\right)$ is
non-degenerate.

In the case when $G$ is a semisimple Lie group of inner type  (in
particular, when $G$ is compact) a classification of regular
generalized complex structures on $G$ is given. We show that any
invariant generalized complex structure on a compact semisimple
Lie group is regular, provided that an additional natural
condition is satisfied.

In the case when $G$ is a semisimple Lie group of outer type, we
describe the subalgebras $\gk$ in terms of appropriate root
subsystems of a root system of $\gg^{\mathbb{C}}$ and we construct
a large class of admissible pairs $(\gk , \omega )$ (hence,
regular generalized complex structures of $G$).

\section{Introduction}

Let $M$ be a smooth manifold. We  will denote by $\mathbb{T}M = TM
\oplus T^*M$ the { \bf generalized  tangent bundle}   of  $M$,
defined  as the direct sum of tangent and cotangent bundles, and
by $g_{\mathrm{can}}$ the canonical indefinite metric on
$\mathbb{T}M$, given by
\begin{equation}\label{gcan}
g_{\mathrm{can}}(X+\xi ,Y+\eta ):= \frac{1}{2}\left( \xi (Y)+\eta
(X)\right) ,\quad\forall X+\xi ,Y+\eta\in\mathbb{T}M.
\end{equation}
A  { \bf generalized almost complex structure }
\cite{thesis,hitchin} on $M$ is a
$g_{\mathrm{can}}$-skew-symmetric field of endomorphisms
$${\mathcal J}:\mathbb{T}M  \rightarrow \mathbb{T}M $$
with ${\mathcal  J}^2 = - \Id $ (where ``$\Id$'' denotes the
identity endomorphism). A generalized almost complex structure
$\mathcal J$ is said to be integrable (or is a {\bf generalized
complex structure}) if  the $i$-eigenbundle $L =
\mathbb{T}^{1,0}M \subset (\mathbb{T}M)^{\mathbb{C}}$ of
$\mathcal{J}$ (called the {\bf holomorphic bundle}  of $\mathcal
J$) is closed under the complex linear extension of the Courant
bracket $[\cdot ,\cdot ]$, defined by
\begin{equation}\label{courant}
[X+\xi ,Y+\eta ] := [X, Y]+ L_{X}\eta -L_{Y}\xi -\frac{1}{2} d
\left( \eta(X) - \xi (Y)\right) ,
\end{equation}
for any smooth sections $X+ \xi$ and $Y + \eta$ of $\mathbb{T}M.$
Complex and symplectic structures define, in a natural way,
generalized complex structures and many definitions and results
from complex and symplectic geometry can be extended to
generalized complex geometry. This paper is concerned with
invariant generalized complex structures on Lie groups.

Section \ref{preliminar}
is intended to fix the notations and conventions we
shall use along the paper.  We recall basic facts we need about
real and complex semisimple Lie algebras, invariant complex structures on
homogeneous manifolds and generalized complex structures on manifolds.
Our approach follows closely
\cite{thesis,knapp,wang}.

In Section \ref{si1} we present an infinitesimal description of
invariant generalized complex structures on Lie groups, in terms
of the so called admissible pairs. The holomorphic bundle $L$ of
an invariant generalized complex structure $\mathcal J$ on a Lie
group $G$, with Lie algebra $\gg$, can be defined in terms of a
pair $(\gk ,\omega)$ (called {\bf $\gg$-admissible}), formed by a
subalgebra $\gk\subset\gg^{\mathbb{C}}$ and a closed $2$-form
$\omega\in\Lambda^{2}(\gk^{*} )$, with the property that
$\omega_{\gl}:= \mathrm{Im}\left(\omega\vert_{\gl}\right)$ is
non-degenerate, where $\gl =\gk\cap\gg$ is the real part of
$\gk\cap\bar{\gk}.$ The general theory of closed forms defined on
Lie algebras was developed in \cite{dmitri}. At the end of this
section we assume that $\gk$ decomposes into an ideal $\gp$ and a
complementary subalgebra $\gs$ and we find necessary and
sufficient conditions for a $2$-form
$\omega\in\Lambda^{2}(\gk^{*})$ to be closed (see Proposition
\ref{rho}). This result gives a distinguished class of admissible
pairs (see Proposition \ref{special}) and will be used repeatedly
in the next sections.

In Section \ref{regularsection} we begin our treatment of
invariant generalized complex structures on semisimple Lie groups.
An invariant generalized complex structure $\mathcal J$ on a
semisimple Lie group $G$ and the associated $\gg$-admissible pair
$(\gk , \omega )$  are called {\bf
regular} if $\gk$ is a regular subalgebra of $\gg^{\mathbb{C}}$, i.e.
is normalized by a Cartan subalgebra $\gh_{\gg}$ of $\gg = \mathrm{Lie}(G)$. We describe all regular
$\gg$-admissible pairs (hence also the associated generalized complex structures on $G$)
where $\gg$ is a real form of inner type and $\gh_{\gg}$ is a maximally
compact Cartan subalgebra of $\gg$ (see Proposition
\ref{corAdmissiblepair} and Theorem \ref{main}).
We study how
these generalized complex structures
can be reduced to the normal form by means of
invariant $B$-field transformations (see Proposition \ref{B-field}). At the end of
this section we show that any invariant generalized complex
structure on a compact semisimple Lie group $G$, with Lie algebra
$\gg$, is regular, provided that the Lie algebra $\gk$ of the
associated $\gg$-admissible pair has the property that $\gk \cap
\gg$ generates a closed subgroup of $G$ (see Theorem
\ref{sumarizat}).

Section \ref{examples}  is concerned with regular generalized complex structures
(or regular admissible pairs $(\gk , \omega )$)
on simple Lie groups $G$ of outer type, using the formalism of Vogan diagrams.
The classification of the subalgebras $\gk$ reduces to the description of the
so called $\sigma$-positive root systems (see Definition \ref{sigmapara}).
We describe explicitly these systems and we obtain a large class of regular generalized
complex structures on $G$.

The results of Sections \ref{regularsection} and \ref{examples}
may be seen as complementary to some previous constructions of
invariant generalized complex or K\"{a}hler structures on some
classes of homogeneous manifolds (e.g. nilpotent or solvable Lie
groups and their compact quotients \cite{ref1,cavalcanti,ref2},
homogeneous manifolds $G/K$, where $G$ is a compact Lie group and
$K$ a closed subgroup of maximal rank, and adjoint orbits in
semisimple Lie algebras \cite{brett}). They also complement a
result of M. Gualtieri, namely that any compact semisimple Lie
group of even dimension admits a $H$-twisted generalized
K\"{a}hler structure (see Example 6.39 of \cite{thesis}).\\

{\bf Acknowledgements.} We are grateful to Paul Gauduchon for his
interest in this work and discussions related to the theory of
real semisimple Lie algebras. We thank the referee for his useful
suggestions and comments and for pointing to us the articles
\cite{ref1,ref2,brett}. This work is partially supported by a
CNCSIS - PN II IDEI Grant, code no. 1187/2008, and Hamburg
University.

\section{Preliminary material }\label{preliminar}

\subsection{Invariant  complex structures on Lie
groups and homogeneous manifolds}\label{liegroups}

\subsubsection{Invariant complex structures on homogeneous manifolds}\label{liegroups}

The Lie algebra of a Lie group will be identified as usual with
the tangent space at the identity element or with the space of
left-invariant vector fields.

Let $G$ be a real Lie group, with Lie algebra $\gg$, and $L$ a
closed connected subgroup of $G$, with Lie  algebra  $\gl$.
Suppose that the space $M= G/ L$ of left cosets is reductive, i.e.
$\gg$ has an $\mathrm{Ad}_{L}$-invariant decomposition
\begin{equation}\label{reductive-deco}
\gg = \gl\oplus \gm .
\end{equation}
We shall identify $\gm$ with the tangent space $T_{o}M$ at the
origin $o= eL \in G/ L.$ An invariant complex structure $J$ on $M$
is determined by its value $J_{o}$ at $o$, which is an
$\mathrm{Ad}_{L}$-invariant complex structure on the vector space
$\gm = T_{o}M$. Let $\gm^{1,0}$ and
$\gm^{0,1}=\overline{\gm^{1,0}}$ be the holomorphic, respectively
anti-holomorphic subspaces of $J_{o}$, so that
\begin{equation}\label{m01}
\gm^{\mathbb{C}} = \gm^{1,0}\oplus\gm^{0,1}.
\end{equation}
The invariance and integrability of $J$ mean that $\gk =
\gl^{\mathbb{C}}+\gm^{1,0}$ is a  (complex) subalgebra of
$\gg^{\mathbb{C}}.$ Conversely, any decomposition (\ref{m01}) of
$\gm^{\mathbb{C}}$ into two vector spaces $\gm^{1,0}$ and
$\gm^{0,1}=\overline{\gm^{1,0}}$ such that $\gk =
\gl^{\mathbb{C}}+\gm^{1,0}$ is a subalgebra of $\gg^{\mathbb{C}}$,
defines an invariant complex structure on $M$. We get the
following well known  algebraic description of invariant complex
structures on reductive homogeneous manifolds.

\begin{prop}\label{deadaugat1}
Let $M= G/ L$ be a reductive homogeneous manifold,  with reductive decomposition
(\ref{reductive-deco}).
There is a  natural one to one correspondence between: \\

i) invariant complex structures on $M$;\\

ii) decompositions $\gm^{\mathbb{C}}= \gm^{1,0}\oplus \gm^{0,1}$,
where  $\gm^{0,1}=\overline{\gm^{1,0}}$ and $\gk =  \gl^{\mathbb{C}}+\gm^{1,0}$ is a subalgebra of $\gg^{\mathbb{C}}$;\\

iii) subalgebras $\gk\subset\gg^{\mathbb{C}}$ such that
\begin{equation}\label{def2}
\gk +\bar{\gk} = \gg^{\mathbb{C}},\quad
\gk\cap\bar{\gk}=\gl^{\mathbb{C}}.
\end{equation}
\end{prop}

In particular, if $M = G$ is a Lie group,
there is a one to one correspondence between invariant complex
structures on $G$ and decompositions
\[
\gg^{\mathbb{C}} = \gg^{1,0}\oplus \gg^{0,1},
\]
where $\gg^{1,0}$ and $\gg^{0,1}$ are subalgebras of
$\gg^{\mathbb{C}}$ and
$\gg^{0,1} = \overline{\gg^{1,0}}.$

\subsubsection{Some basic facts  on semisimple Lie algebras}

We fix our notations from the theory of semisimple Lie algebras.\\

{\bf Complex semisimple Lie algebras.} Recall that any complex
semisimple Lie algebra $\gg^{\mathbb C}$ has a root space
decomposition
\[
\gg^{\mathbb{C}} = \gh +\gg (R)= \gh + \sum_{\alpha\in
R}\gg_{\alpha}
\]
with respect to a Cartan subalgebra $\gh$, where $R \subset \gh^*$
is the  root system of $\gg^{\mathbb C}$ relative to $\gh$ and for
any subset $P \subset R$ we denote by $\gg(P) $  the direct sum of
root spaces
\[
\gg(P) := \sum_{\alpha \in P} \gg_{\alpha}.
\]
Along the paper we will denote by $E_\alpha \in  \gg_\alpha$  the
root vectors of a Weyl basis of $\gg(R)$, which have the following
properties:
\begin{enumerate}
    \item[i)]
\[
\langle E_{\alpha}, E_{-\alpha}\rangle =1,\quad\forall\alpha \in
R,
\]
where $B= \langle\cdot , \cdot\rangle$  is the  Killing  form of $\gg^{\mathbb C}$; \item[ii)] the structure constants
$N_{\alpha\beta}$ defined by
\[
[E_{\alpha}, E_{\beta}]= N_{\alpha\beta} E_{\alpha
+\beta},\quad\forall \alpha ,\beta \in R
\]
are real and satisfy
\[
N_{-\alpha ,-\beta} = - N_{\alpha \beta},\quad\forall \alpha
,\beta \in R.
\]
\end{enumerate}
We will identify  the  dual space  $\gh^*$  with $\gh$ using  the
restriction  $\langle\cdot, \cdot\rangle$ of the Killing form to
$\gh$, which is non-degenerate on $\gh$ and positive definite on
the real form $\gh(\mathbb{R})$ of $\gh$, spanned by the root
system $R \subset \gh$. For a set of roots $P$, we will constantly
use the notation $P^{\mathrm{sym}}:=P \cap (-P)$ for the symmetric
part of $P$ and $P^{\mathrm{asym}} := P \setminus
P^{\mathrm{sym}}$  for
the anti-symmetric part.\\

{\bf Real semisimple Lie algebras.}  Let $\gg$ be  a real
semisimple Lie algebra. Its complexification  $\gg^{\mathbb{C}}$
is a complex  semisimple Lie algebra and $\gg = (\gg^{\mathbb
C})^{\sigma}$, that is, $\gg$ can be reconstructed from
$\gg^{\mathbb{C}}$ as the fix point set of a {\bf complex
conjugation} or {\bf antiinvolution}
\[
\sigma :\gg^{\mathbb C} \rightarrow \gg^{\mathbb C},\quad
x\rightarrow \bar{x},
\]
that is, $\sigma$ is an involutive automorphism of
$\gg^{\mathbb{C}}$, as a real Lie algebra, and is antilinear. One
can always assume that the antiinvolution $\sigma$ preserves a
Cartan subalgebra $\gh $ of $\gg^{\mathbb{C}}$. Then $\gh_{\gg} =
\gh^\sigma := \gh \cap \gg $ is a  Cartan subalgebra of $\gg$. The
antiinvolution $\sigma$ acts naturally on the  set of roots $R$ of
$\gg^{\mathbb{C}}$ relative to $\gh$, by
\[
\sigma (\alpha ):= \overline {\alpha\circ\sigma },\quad \forall
\alpha\in R
\]
and
\[
\sigma (\gg_{\alpha} )= \gg_{\sigma (\alpha )},\quad\forall\alpha\in R.
\]
The {\bf compact  real form} of a complex semisimple Lie algebra
$\gg^{\mathbb{C}}$ is unique (up to conjugation) and is defined by
the antiinvolution $\tau$ (called  {\bf compact}) given  by
\[
\tau |_{\gh({\mathbb R})} =-\Id,  \,\,  \tau (E_\alpha )= -
E_{-\alpha},\, \forall \alpha \in R.
\]
Assume now that $\gg$ is a real form of $\gg^{\mathbb C}$, not
necessarily compact. Consider a Cartan decomposition
\[
\gg = \gk \oplus \gp
\]
of $\gg$ and let $\theta$ be the associated Cartan involution,
with $+1$ eigenspace $\gk$ and $-1$ eigenspace $\gp .$ The real
form
\[
\gu = \gk \oplus i \gp
\]
is compact, the corresponding antiinvolution $\tau$ commutes with
$\theta$ and $\gg = (\gg^{\mathbb C})^{\sigma}$, where $\sigma : =
\theta \circ \tau$. A $\theta$-invariant Cartan subalgebra
$\gh_{\gg}$ of $\gg$ decomposes as
\[
\gh_{\gg} = \gh^{+} \oplus \gh^{-},\quad \gh^{+}\subset \gk ,\quad
\gh^{-}\subset \gp
\]
and any root of $\gg^{\mathbb C}$ relative to
$(\gh_{\gg})^{\mathbb C}$ takes purely imaginary values on
$\gh^{+}\oplus i \gh^{-}.$ The Cartan subalgebra $\gh_{\gg}$ is
called maximally compact (respectively, maximally non-compact) if
the dimension of $\gh^{+}$ (respectively, $\gh^{-}$) is as large
as possible. The real form $\gg$ is of {\bf inner type} if
$\theta$ is an inner automorphism of $\gg$, that is, it belongs to
$\mathrm{Int}(\gg)$. The definition is independent of the choice
of $\theta$, since any two Cartan involutions of $\gg$ are
conjugated via $\mathrm{Int}(\gg )$, see e.g.
\cite[p.\,301]{knapp}. It can be shown that $\gg$ is of inner type
if and only if there is a (maximally compact) Cartan subalgebra $\gh_{\gg}$ included
in $\gk$, see e.g. \cite[p.\,24]{burstall}. Then
\[
\sigma (\alpha ) = - \alpha
\]
for any root $\alpha$ of $\gg^{\mathbb{C}}$ relative to
$(\gh_{\gg})^{\mathbb C}.$ Any compact real form is a real form of inner type.\\

A non-inner real form $\gg$ (and the corresponding antilinear
involution) is called {\bf outer}. For such a real form, the
Dynkin diagram of $\gg^{\mathbb{C}}$ admits a non-trivial
automorphism of order two, defined by the Cartan involutions of
$\gg$ (more precisely, any Cartan involution $\theta$ permutes a
set of simple roots of $\gg^{\mathbb C}$ relative to
$(\gh_{\gg})^{\mathbb C}$, where $\gh_{\gg}\subset \gg$ is a
maximally compact $\theta$-invariant Cartan subalgebra, giving
rise to the resulting automorphism of the Dynkin diagram, see
\cite[p.\,339]{knapp}). In particular, if $\gg^{\mathbb C}$ is
simple, then $\gg^{\mathbb{C}} = A_{n}$, $(n\geq 2$), $D_{n}$
$(n\geq 3)$ or $E_{6}.$ The list of simple non-complex Lie
algebras of outer type is short and it is given below
(see \cite{knapp}, Appendix C):\\
$$ \gsl_n(\mathbb{R})\, (n>2),\, \gsl_n(\mathbb{H})\, (n\geq 2),\, \{ \gso_{2p+1,2q+1},\,
0\leq p\leq q\}\setminus \{ \gso_{1,1},\gso_{1,3}\},\,
\ge_6(\gf_4),\, \ge_6(\gsp_4).$$ For the real forms of the
exceptional Lie algebra $\ge_6$ we indicate in the bracket the
type of maximal compact subalgebras. A Lie group is called {\bf inner} (respectively,
{\bf outer})  if its Lie algebra is inner (respectively, outer).

\subsubsection{Invariant complex structures on homogeneous
manifolds of compact semisimple Lie groups}\label{wangsection}

Let $F = G/K = {\mathrm Ad}_G(x_0) \subset \gg $ be  a flag manifold
of a compact connected semisimple Lie group $G$,
with Lie algebra $\gg = (\gg^{\mathbb{C}})^{\tau}$,
and $K = T\cdot L$ a decomposition of  the stabilizer
$K$ into a product of a  torus
$T$  and a semisimple group $L$.
Let $T^{\prime}$ be a subtorus of $T$, such that the quotient $T/ T^{\prime}$ has even
dimension.  The total space of the fibration
\[
 \pi:M = G/(L\cdot T^{\prime}) \rightarrow F= G/K
\]
admits an invariant complex structure, defined  by an invariant  complex structure on the fiber $T/ T^{\prime}$  and an invariant
complex structure on the base $F$, such that $\pi$ is an holomorphic fibration.
Moreover, any invariant complex structure on $M$ is obtained in this way.

More precisely,  let $\gh_{\gg}$ be
a Cartan subalgebra of $\gg$ included in $\gk =\mathrm{Lie}(K).$
The complexification of the $B$-orthogonal reductive decomposition
$\gg = \gk \oplus \gp$
of $F = G/K$ is given by
\[
\gk^\mathbb{C} = \gh + \gg([\Pi_0]),\quad \gp^{\mathbb{C}} = \gg(R^{\prime}),
 \quad R^{\prime} = R \setminus [\Pi_0]
\]
where  $\Pi_0  \subset \Pi $  is a subset  of a system
$\Pi$ of simple roots of $\gg^{\mathbb{C}}$  relative to
$\gh = (\gh_{\gg})^{\mathbb{C}}$ and
$[\Pi_0]:= R \cap {\mathrm{span}}(\Pi_0 )$
is the set of all roots from $R$ spanned by $\Pi_{0}.$
The Cartan subalgebra $\gh_{\gg}$ decomposes into a
$B$-orthogonal direct sum
\[
\gh_{\gg} = \gt \oplus \gh_0  = \gt^{\prime} \oplus \mathfrak{a}\oplus \gh_{0}
\]
where $\gt = \mathrm{Lie} (T )$ and $\gt^{\prime} =\mathrm{Lie}(T^{\prime}).$
The complexification of the  reductive decomposition of $ M= G/(L\cdot T^{\prime}) $ is
given by
\[
 \gg^{\mathbb{C}} =  (\gl +\gt^{\prime})^{\mathbb{C}} \oplus\gm^{\mathbb{C}}
\]
where $\gl =\mathrm{Lie}(L)$,
$\gl^{\mathbb{C}}=  \gh_0^{\mathbb{C}} + \gg([\Pi_{0}])$
and  $\gm^{\mathbb{C}}=\mathfrak{a}^{\mathbb{C}} + \gg(R^{\prime}).$

 \begin{thm}\label{wang} \cite{wang} \begin{enumerate}
\item[i)] The  compact homogeneous manifold $M = G/(L\cdot T^{\prime})$ described above  admits
an invariant  complex structure defined by  the decomposition
\[
\mathfrak{m}^{\mathbb{C}}= \mathfrak{m}^{1,0}\oplus
\overline{\mathfrak{m}^{1,0}}= \mathfrak{m}^{1,0}\oplus
\tau(\mathfrak{m}^{1,0}),
\]
where
\[
\mathfrak{m}^{1,0}= {\mathfrak{a}}^{1,0}+ \gg
(R^{\prime}_{+}), \quad R^{\prime}_+ := R^+ \cap (R \setminus
[\Pi_0]),
\]
$R^{+}$ is the positive root system defined by $\Pi$
and $\mathfrak{a}^{1,0}$ is  the holomorphic  subspace of a
complex structure  $J^{\mathfrak{a}}$ on $\mathfrak{a} $.

\item[ii)] Conversely, any invariant complex structure on an
homogeneous manifold of a compact, connected, semisimple
Lie group $G$ can be obtained by this construction.

\end{enumerate}

\end{thm}

\subsection{Generalized complex structures on
manifolds}\label{prel1}

In this paper we consider only generalized  complex structures of
constant type.

Let $\mathcal J$ be a generalized almost complex structure on an
$n$-dimensional manifold $M$. According to M. Gualtieri
\cite{thesis},
 the
 holomorphic bundle $L=\mathbb{T}^{1,0}M \subset (\mathbb{T}M)^{\mathbb{C}}$ of $\mathcal
J$ can be described in terms of a subbundle $E\subset
(TM)^{\mathbb{C}}$ and a $2$-form $\tilde{\omega}\in
\Gamma(\Lambda^{2}(E^{*}))$ defined on $E$. We now recall this
description.

Note that $L\subset (\mathbb{T}M)^{\mathbb{C}}$ is isotropic with
respect to the complex linear extension of $g_{\mathrm{can}}$ and
 $L\oplus \bar{L}= (\mathbb{T}M)^{\mathbb{C}}.$
Conversely, any isotropic subbundle $L\subset
(\mathbb{T}M)^{\mathbb{C}}$ such that
 $L\oplus \bar{L}= (\mathbb{T}M)^{\mathbb{C}}$
defines a generalized almost complex structure $\mathcal J$, whose
complex linear extension to $(\mathbb{T}M)^{\mathbb{C}}$ satisfies
${\mathcal J}\vert_{L}= i$ and ${\mathcal J}\vert_{\bar{L}}=-i.$
These considerations play a key role in the proof of the following
Proposition:

\begin{prop}\label{cheie}\cite[p.\,49]{thesis} A complex rank $n$ subbundle $L$ of $(\mathbb{T}M)^{\mathbb{C}}$
is the  holomorphic bundle of a
generalized almost complex structure ${\mathcal J}$ if and only if
it is of the form
\[
L= L(E, \tilde{\omega} ):= \{ X+\xi\in E\oplus (T^{\mathbb{C}}M)^{*},
\quad \xi\vert_{E}= \tilde{\omega} (X,\cdot )\}
\]
where $E\subset (TM)^{\mathbb{C}}$ is a complex
subbundle and
$\tilde{\omega}\in \Gamma(\Lambda^{2}(E^{*}))$ is a complex  2-form on $E$   such that
the imaginary  part $\mathrm{Im}\left(\tilde{\omega}\vert_{\Delta}\right)$ is non-degenerate.
Here
\[
\Delta = E \cap TM  \subset TM
\]
is the real part of $E\cap\bar{E}$, i.e.
\[
\Delta^{\mathbb{C}}=E\cap\bar{E}.
\]
Moreover, ${\mathcal J}$ is integrable if and only if $E$ is
involutive (i.e. its space of sections is closed under the Lie
bracket) and $d_{E}\tilde{\omega} =0$, where $d_{E}$ is the
exterior derivative along $E$.
\end{prop}

The codimension of the subbundle $E \subset  T^{\mathbb C}M$ is called
the {\bf type} of the generalized complex structure $\mathcal
J.$

Any complex or symplectic structure defines a generalized complex
structure (see e.g. \cite{thesis}). Other examples of generalized
complex structures can be obtained using $B$-field transformations,
as follows. Any closed $2$-form $B\in\Omega^{2}(M)$ (usually
called a {\bf $B$-field}) defines an automorphism of $\mathbb{T}M$,
by
\[
\mathrm{exp}(B)(X+\xi ) = X + i_{X}B +\xi ,\quad \forall X+\xi
\in\mathbb{T}M
\]
which preserves the Courant bracket (this follows from $dB=0$). If
$\mathcal J$ is a generalized complex structure on $M$, with
holomorphic bundle $L(E, \tilde{\omega})$, then $L(E, \tilde{\omega} +i^{*}B
)$, where $i^{*}B\in\Lambda^{2}(E^{*})$ is the restriction of (the complexification of) $B$ to $E$, is the
holomorphic bundle of another generalized complex structure
$\mathrm{exp}(B)\cdot {\mathcal J}$, called the {\bf $B$-field
transformation of $\mathcal J$}. Obviously,
 a $B$-field transformation preserves  the
type.

The last notion we need from generalized complex geometry is the
normal form of generalized complex structures \cite[p.\,1507]{ornea}.
Recall first that an (almost) $f$-structure on a manifold $M$ is an
endomorphism $F$ of $TM$ satisfying $F^{3}+ F =0.$ Let $T^{0}M$,
$T^{1,0}M$ and $T^{0,1}M$ be the eigenbundles of the complex
linear extension of $F$, with eigenvalues $0$, $i$ and $-i$
respectively. A (real) $2$-form $\eta\in\Omega^{2}(M)$ is called
compatible with $F$ if $\eta_{\mathbb C}\vert_{T^{0}M}$ is
non-degenerate and $\mathrm{Ker}(\eta_{\mathbb C} ) =
T^{1,0}M\oplus T^{0,1}M$, where $\eta_{\mathbb C}$ is the
complex linear extension of $\eta .$ A generalized (almost)
complex structure $\mathcal J$ on $M$, with holomorphic bundle
$L$, is in {\bf normal form} if $L= L(T^{0}M\oplus T^{1,0}M,
i\eta_{\mathbb C} )$ for some almost $f$-structure and
compatible $2$-form $\eta$. In the language of \cite{thesis},
this means that ${\mathcal J}_{p}$, at any $p\in M$,  is the
direct sum of a complex structure and a symplectic structure.

\section{Infinitesimal description of invariant generalized complex structures on Lie
groups}\label{si1}

\subsection{Admissible pairs: definition and general results}

Generalized geometry of Lie groups and homogeneous spaces already
appears in the literature (see e.g.
\cite{ref1,cavalcanti,ref2,brett}). It is known that  invariant
generalized complex structures on a Lie group $G$ are in bijective
correspondence with invariant complex structures on the cotangent
group $T^{*}G$, which are compatible with the standard neutral
bi-invariant metric of $T^{*}G$ \cite[p.\,771]{ref1}. The proof
follows from the remark that the restriction of  the Courant
bracket to the space of invariant generalized vector fields (i.e.
invariant sections of $\mathbb{T}G$) coincides with the Lie
bracket in the Lie algebra $\gt^{*}(\gg )$ of the cotangent group
$T^{*}G.$

In this section we develop an infinitesimal description of
invariant generalized complex structures on Lie groups in terms of
the so called admissible pairs, which will be a main tool in this
paper.\footnote{ We thank the referee for pointing to us article
\cite{brett}, where an infinitesimal description of invariant
generalized complex structures on homogeneous manifolds $G/K$, in
terms of the so called generalized complex pairs (GC-pairs), was
developed. When $K=\{e\}$, a GC-pair is an admissible pair.}

Let $G$ be a real Lie group. On the generalized tangent bundle
$\mathbb{T}G$ we consider the natural action induced by left translations:
\[
g \cdot (X + \xi) := {(L_g)}_* X + \xi \circ (L_g^{-1})_{*}
,\quad\forall g\in G,\quad \forall X+\xi \in\mathbb{T}G.
\]

\begin{defn} A generalized almost complex structure $\mathcal J$ on $G$
is called {\bf invariant} if
\begin{equation}\label{sting}
{\mathcal J}(X+\xi )= g^{-1}\cdot {\mathcal J}\left( g\cdot
(X+\xi)\right) ,\quad\forall g\in G, \quad \forall X+\xi\in
\mathbb{T}G.
\end{equation}
\end{defn}

The holomorphic bundle $L = L(E, \tilde{\omega} )$ of an invariant generalized
almost complex structure $\mathcal J$ is determined by its fiber $L_{e}$
at the identity element $e\in G$,  i.e.
by a subspace  $\gk:= E_{e} \subset \gg^{\mathbb{C}}$
and a complex $2$-form $\omega := \tilde{\omega}_{e}\in \Lambda^{2}(\gk^{*})$.
It is easy to check that the bundle $E$ is involutive if and only if $\gk$ is a complex subalgebra of $\gg^{\mathbb{C}}$.
Moreover, if $E$ is involutive, then
$d_{E}\tilde{\omega} =0$ if and only if $\omega$ is closed as a $2$-form on the Lie algebra $\gk$, i.e. for any $X, Y, Z\in \gk $,
\[
(d_{\gk}\omega )(X, Y, Z) :=\omega ( X, [Y,Z]) +\omega (Z, [X, Y])
+\omega (Y, [Z, X]) =0.
\]
This last statement follows from the following computation:  for any $X, Y, Z\in E_{e}$ (viewed
either as left invariant sections of $E$ or elements of $\gk$),
\begin{align*}
(d_{E}\tilde{\omega})_{e}(X, Y, Z) &= X\left(\tilde{\omega}(Y, Z)\right) - Y\left( \tilde{\omega}(X, Z)\right)
+ Z\left( \tilde{\omega}(X, Y)\right)\\
& + \tilde{\omega} ( X, [Y,Z]) +\tilde{\omega} (Z, [X, Y])
+\tilde{\omega} (Y, [Z, X])\\
&= (d_{\gk}\omega )(X, Y, Z),
\end{align*}
where we used that $\tilde{\omega} (Y, Z)$ and the remaining similar terms are constant (from the  invariance of $\tilde{\omega}$).\\

To simplify terminology we introduce the following definition.

\begin{defn} Let $\gg$ be a real Lie algebra, given by the fixed point set of an antiinvolution
$\sigma (x) = \bar{x}$ of $\gg^{\mathbb{C}}.$
A pair $(\gk ,\omega )$ formed by a complex subalgebra $\gk\subset \gg^{\mathbb{C}}$ and a
closed $2$-form $\omega\in\Lambda^{2}(\gk^{*})$ is called {\bf $\gg$-admissible} if\\

i) $\gg^{\mathbb{C}} = \gk +\bar{\gk}$;\\

ii) $\omega_{\gl}:=\mathrm{Im}\left(\omega\vert_{\gl}\right)$ is a
symplectic form on $\gl = \gk\cap\gg$, i.e. it  is non-degenerate
(and closed).

\end{defn}

The following result, which is a corollary of  Proposition
\ref{cheie}, reduces the  description of invariant generalized
complex structures on a Lie group $G$ to the description of
$\gg$-admissible pairs.

\begin{thm}\label{20}
Let $G$ be a Lie group, with Lie algebra $\gg .$
There is a natural one to one correspondence between:

\begin{enumerate}
\item[i)]  invariant generalized complex structures on $G$;

\item[ii)] $\gg$-admissible pairs $(\gk ,\omega )$.

\end{enumerate}

More precisely, a $\gg$-admissible pair $(\gk ,\omega )$ defines an invariant
generalized complex structure $\mathcal J$, with holomorphic space at
$e\in G$ given by
\[
{\mathbb T}_{e}^{1,0}G= \gw^{1,0}:= \{ X+\xi\in \gk\oplus
(\gg^{\mathbb{C}})^{*}:\quad \xi\vert_{\gk}= \omega (X,\cdot )\}.
\]
\end{thm}

Theorem \ref{20} has the following important consequence.

\begin{cor}\label{quot} Let $\mathcal J$ be an invariant
generalized complex structure on a Lie group $G$, defined by a
$\gg$-admissible pair $(\gk ,\omega ).$ Suppose that the real Lie
algebra
\[
\gl =\gg \cap \gk \subset \gg
\]
generates a closed, connected Lie subgroup $L$ of $G$, such that
the homogeneous space $M= G/ L$ is reductive. Then $\mathcal J$
defines an invariant complex structure $J$ on  $M$.
\end{cor}

\begin{proof}
Since $\gk$ belongs to an admissible pair, $\gk +
\bar{\gk}=\gg^{\mathbb{C}}.$ Moreover, $\gk\cap\bar{\gk}=
\gl^{\mathbb{C}}.$ From Proposition \ref{deadaugat1} {\it iii)},
$\gk$ defines an invariant complex structure $J$ on $M$.

\end{proof}

We end this section with a property of admissible pairs, which
will be useful in our treatment of invariant generalized complex
structures on semisimple Lie groups.

\begin{prop}\label{n}
Let $\gg$ be a Lie algebra and $(\gk ,\omega )$ a $\gg$-admissible pair.
Suppose that $\gl = \gk\cap\gg$ is reductive. Then $\gl$ is
abelian.
\end{prop}

Proposition \ref{n} is a consequence of the following general statement
(applied to the reductive Lie algebra $\gl = \gk\cap \gg$, which
admits a symplectic form, namely $\omega_{\gl}: =\mathrm{Im}\left(
\omega\vert_{\gl}\right)$):

\begin{lem}\label{reductive-lem} Any (real or complex) reductive Lie
algebra which admits a symplectic form is abelian.
\end{lem}

\begin{proof}
Let $\gl= \gc \oplus \gl^{s}$ be a reductive Lie algebra, with the
center $\gc$ and semisimple part $\gl^{s}.$ Suppose that $\omega
\in \Lambda^{2}(\gl^{*} )$ is a non-degenerate closed $2$-form on
$\gl$. Then $\gc$ is $\omega$-orthogonal to $\gl^s$. Indeed,
 for $Y, Z\in \gl^{{s}}$ and $X\in \gc$ , we get
 \[
0 =  d\omega (X, Y, Z) = \omega (X, [Y, Z]) +\omega ( Z, [X, Y])+\omega
( Y, [Z, X])=  \omega (X, [Y, Z]).
\]
Since $[\gl^s, \gl^s] = \gl^s$, the claim follows. Assume now that
$\gl^s \neq 0$. Then $\omega^s := \omega|_{\gl^s}$ is a
non-degenerate closed 2-form on  the semisimple Lie algebra
$\gl^s$. Being closed, $\omega^{s}$ is the differential of a
1-form $\xi \in (\gl^s)^*$
(see e.g. \cite[p.\,16]{oni}),  which is dual with respect to the
Killing form $\langle\cdot , \cdot\rangle$ of $\gl^{s}$ to a vector $X_0\in \gl^{s}$:
\[
\xi(X) = \langle X_0,X\rangle , \, \, \forall X \in \gl^s.
\]
For any $ Y \in \gl^s$ we  get
\[
\omega (X_{0}, Y) = d\xi (X_{0}, Y) = -\xi ([X_{0}, Y]) = -
\langle X_{0}, [X_{0}, Y]\rangle =  \langle [X_0,X_0], Y\rangle  =
0 .
\]
Hence  $X_0$ belongs to the kernel of $\omega^s$.
Since $\omega^{s}$ is non-degenerate,  $X_{0}=0$ and hence $\gl^{s}=0$.
We obtain a contradiction.

 \end{proof}

\subsection{A class of admissible pairs}

We have reduced  the description of invariant generalized complex
structures on a Lie group $G$, with Lie algebra $\gg$,  to the
description of $\gg$-admissible pairs $(\gk , \omega )$, and, in
particular, to the description of closed $2$-forms on $\gk .$ In
this section we find a special class of $\gg$-admissible pairs
$(\gk, \omega)$ (see Proposition \ref{special} below). We begin
with the following result which will be used repeatedly in our
description of invariant generalized complex structures on
semisimple Lie groups.

\begin{prop}\label{rho} Suppose that a Lie algebra $\gk$ admits a
semidirect  decomposition
\[
\gk = \gs \oplus \gp
\]
into a subalgebra $\gs$ and an ideal $\gp .$ Decompose a 2-form
$\rho \in \Lambda^2(\gk^*)$ on $\gk$   into three parts
\[
\rho = \rho_{0} +\rho_{1}+\rho_{2},
\]
where   $\rho_{0}\in \Lambda^{2}(\gs^{*})$ is the $\gs$-part,
$\rho_{1}\in\Lambda^{2}(\gp^{*})$ is the $\gp$-part  and
$\rho_{2}\in \gs^{*}\wedge \gp^{*}\subset\Lambda^{2}(\gk^{*})$  is
the mixed part of $\rho$ (all trivially extended  to $\gk$).
Then  the form $\rho$ is closed if and only if the following
conditions are satisfied:

\begin{enumerate}

\item the forms $\rho_0, \rho_1$  are closed  on $\gs$ and,
respectively, on $\gp$;

\item the following two conditions are satisfied:
\begin{equation}\label{c1}
\rho_{2}(s, [p, p^{\prime}])  = \rho_{1}([s, p], p^{\prime})
+\rho_{1}(p, [s,p^{\prime}])
\end{equation}
and
\begin{equation}\label{c2}
\rho_{2}([s, s^{\prime}], p) + \rho_{2}([s^{\prime}, p], s)
+\rho_{2}([p, s], s^{\prime})=0,
\end{equation}
for any $s, s^{\prime}\in \gs$ and $p, p^{\prime}\in \gp .$
\end{enumerate}

In particular, if $\gs$ is $\rho$-orthogonal to $\gp$, i.e. $\rho_{2}=0$,
then $\rho $  is  closed if and only if  its $\gs$-part $\rho_0$ is closed and
its  $\gp$-part  $\rho_{1}$ is closed   and
$\mathrm{ad}_{\gs}$-invariant.
\end{prop}

\begin{proof}
The proof is straightforward.
\end{proof}

\begin{prop}\label{special}
Let $\gg$ be a real Lie algebra and $\sigma (x) = \bar{x}$ the
associated antiinvolution on $\gg^{\mathbb C}$. Suppose that a
complex subalgebra $\gk\subset \gg^{\mathbb{C}}$, with
$\gg^{\mathbb{C}}= \gk + \bar{\gk}$, has an ideal $\gp$,
complementary to the subalgebra $\gl^{\mathbb{C}}=\gk
\cap\bar{\gk}$, that is,
\[
\gk = \gl^{\mathbb{C}}\oplus \gp ,
\]
Then any (complex) $2$-form
$\omega\in\Lambda^{2}(\gk^{*})$ which defines  a $\gg$-admissible
 pair $(\gk ,\omega)$  is given by
\[
\omega = \widehat{\omega}_{0}+\rho_{1}+\rho_{2},
\]
where $\widehat{\omega}_{0}\in \Lambda^{2}(\gl^{\mathbb{C}})^{*}$
is a closed $2$-form on $\gl^{\mathbb{C}}$, such that $\mathrm{Im}\left( \widehat{\omega}_{0}\vert_{\gl}\right)$
is non-degenerate,  $\rho_{1}\in\Lambda^{2}(\gp^{*})$ is a closed 2-form on $\gp$ and
$\rho_{2}\in (\gl^{\mathbb C})^{*}\wedge \gp^{*}$ is a  2-form
which satisfies the following two conditions :
\[
\rho_{2}(l, [p, p^{\prime}])  = \rho_{1}([l, p], p^{\prime})
+\rho_{1}(p, [l,p^{\prime}])
\]
and
\[
\rho_{2}([l, l^{\prime}], p) + \rho_{2}([l^{\prime}, p], l)
+\rho_{2}([p, l], l^{\prime})=0,
\]
for any $l, l^{\prime}\in \gl^{\mathbb{C}}$ and $p, p^{\prime}\in
\gp .$
\end{prop}

\section{Invariant generalized complex structures of regular type on
semisimple Lie groups}\label{regularsection}

In the remaining part of the paper we  define and study a class (called regular) of invariant
generalized complex structures on semisimple Lie groups.

\subsection{Regular $\gg$-admissible pairs: definition and general
results}\label{regularpairs}

Let $G$ be a real semisimple Lie group with Lie algebra $\gg$.

\begin{defn} A $\gg$-admissible pair $(\gk ,\omega )$ and
the associated  invariant  generalized complex structure $\mathcal
J$ on $G$ are called {\bf regular} if the complex  subalgebra
$\gk\subset\gg^{\mathbb{C}}$  is {\bf regular}, i.e. it is
normalized by a Cartan subalgebra of $\gg$.
\end{defn}

We denote by
\[
\sigma :\gg^{\mathbb{C}}\rightarrow\gg^{\mathbb{C}},\quad x\rightarrow\bar{x}
\]
the conjugation of $\gg^{\mathbb{C}}$ with respect to $\gg$. Let
$\gk$ be a regular subalgebra of $\gg^{\mathbb{C}}$. The
complexification  $\gh$ of the Cartan subalgebra $\gh_{\gg}$ of
$\gg$ which normalizes $\gk$ is a $\sigma$-invariant Cartan
subalgebra of $\gg^{\mathbb{C}}.$ Being regular, the subalgebra
$\gk$ is of the form
\begin{equation}\label{defk}
\gk = \gh_{\gk} + \gg (R_{0})
\end{equation}
where $\gh_{\gk}:=\gk \cap \gh$ and $R_{0}\subset R$ is a closed
subset of the root system $R$ of $\gg^{\mathbb{C}}$ relative to
$\gh .$ The condition $\gk +\bar{\gk}= \gg^{\mathbb{C}}$ from the
definition of admissible pairs is equivalent to
\[
R_{0}\cup \sigma (R_{0})= R,\quad \gh_{\gk}+\bar{\gh}_{\gk} = \gh
.
\]
To simplify terminology we introduce the following definition:

\begin{defn}\label{sigmapara} i) A subset $R_{0}\subset R$ is called {\bf
$\sigma$-parabolic} if it is closed and
\[
R_{0}\cup \sigma (R_{0}) = R.
\]
ii) A $\sigma$-parabolic subset $R_{0}\subset R$ is called a {\bf
$\sigma$-positive system} if it satisfies the additional condition
\[
R_{0}\cap \sigma (R_{0}) =\emptyset .
\]
iii) Two $\sigma$-parabolic subsets $R_0, R_0'$ are called {\bf
equivalent} if one of  them  can be  obtained  from the other by
transformations
 $R \to - R, \,\,  R \to \sigma(R)$  and a transformation  from the Weyl
  group of $R$, which commutes with $\sigma$.
 \end{defn}

We remark that if $\gg$ is a real form
of inner type of $\gg^{\mathbb{C}}$ and $\gh_{\gg}$ is a maximally
compact Cartan subalgebra of $\gg$, then
\begin{equation}\label{conditiegen}
\sigma (\alpha ) = -\alpha ,\quad\forall\alpha\in R
\end{equation}
and by a result of Bourbaki \cite{bourbaki} (Chapter VI, Section 1.7), $\sigma$-parabolic
subsets (respectively, $\sigma$-positive systems) of $R$ are just
parabolic subsets, that is, closed subsets which contain a
positive root system (respectively, positive root
systems).

\begin{lem}\label{regk} Let $\gk$ be the regular subalgebra
(\ref{defk}) of
$\gg^{\mathbb{C}}$, such that
\begin{equation}\label{suplimentara} \left( R_{0}\cap \sigma
(R_{0})\right)^{\mathrm{asym}} = \emptyset .
\end{equation}
Suppose that $\gk$ can be included into a
$\gg$-admissible pair $(\gk, \omega)$. Then
\[
\gk = \gl^{\mathbb{C}} + \mathfrak{a}^{1,0} + \gg (R_{0})
\]
where $R_{0}$ is a $\sigma$-positive system of $R$,
 $\gl := \gk \cap\gh_{\gg}$ and
$\mathfrak{a}^{1,0}$ is the holomorphic space of a complex
structure $J^{\mathfrak{a}}$ on a complement $\mathfrak{a}$ of
$\gl$ in $\gh_{\gg}$. In particular, the dimension of
$\mathfrak{a}$ is even.
\end{lem}

\begin{proof} The complex conjugated subalgebra $\bar{\gk}$ has the form
\begin{equation}\label{gk-ad}
\bar{\gk}= \bar{\gh}_{\gk}+ \sigma (\gg(R_{0}) )
=\bar{\gh}_{\gk}+\gg (\sigma (R_{0})).
\end{equation}
From (\ref{gk-ad}),
\[
\gk \cap\bar{\gk} = \gh_{\gk}\cap\bar{\gh}_{\gk} +\gg (
R_{0}\cap\sigma (R_{0})).
\]
Condition (\ref{suplimentara}) means that $R_{0}\cap\sigma
(R_{0})$ is symmetric. Thus the  Lie algebra $\gk\cap\bar{\gk}$ is
reductive,  with semisimple part generated by $\gg
(R_{0}\cap\sigma (R_{0}))$ and the center which is the annihilator
of $R_{0}\cap\sigma (R_{0})$ in $\gh_{\gk}\cap\bar{\gh}_{\gk}$.
Since $\gk \cap \bar{\gk}$ is a reductive subalgebra with a
symplectic form, by  Lemma \ref{reductive-lem} it is commutative.
It follows that $R_{0}\cap \sigma (R_{0}) =\emptyset .$ On the
other hand, $\gk + \bar{\gk }= \gg^{\mathbb{C}}$ implies that
$R_{0}\cup \sigma (R_{0}) = R.$ We proved that $R_{0}$ is a
$\sigma$-positive system.

Let $\gw$ be a complement of $\gh_{\gk}\cap\bar{\gh}_{\gk}$ in
$\gh_{\gk}$. Since $\gh_{\gk} +\bar{\gh}_{\gk} =\gh$, $\gw
+\bar{\gw} = \mathfrak{a}^{\mathbb{C}}$ where $\mathfrak{a}$ is a
complement of $\gl$ in $\gh_{\gg}.$ Being transverse, $\gw$ and
$\bar{\gw}$ are the holomorphic and anti-holomorphic spaces of a
complex structure $J^{\mathfrak{a}}$ on $\mathfrak{a}.$

\end{proof}

Note that if $\sigma (\alpha ) = -\alpha$, for any $\alpha\in R$,
the condition (\ref{suplimentara}) is automatically satisfied.
Our main result in this section is stated as follows.

\begin{prop}\label{corAdmissiblepair}
Let $\gg$ be a real form of inner type of a
complex semisimple Lie algebra $\gg^{\mathbb C}$. Any regular
subalgebra of $\gg^{\mathbb C}$ which is normalized by a maximally
compact Cartan subalgebra $\gh_{\gg}$ of $\gg$ and can be included
in a $\gg$-admissible pair is of the form
\begin{equation}\label{defalgk}
\gk = \gl^{\mathbb{C}}+ \mathfrak{a}^{1,0} + \gg (R^{+})
\end{equation}
where $\gl := \gk\cap \gh_{\gg}$, $\mathfrak{a}^{1,0}$ is the
holomorphic space of a complex structure on a complement
$\mathfrak{a}$ of $\gl$ in $\gh_{\gg}$ and $R^{+}$ is a system of
positive roots of $\gg^{\mathbb{C}}$ relative to $\gh :=
(\gh_{\gg})^{\mathbb C}.$
\end{prop}

In the next section we find all $2$-forms, which, together with
the subalgebra (\ref{defalgk}), form $\gg$-admissible pairs.

\subsection{Regular pairs on semisimple Lie groups
of inner type}\label{lace}

In this Section we assume that the real form $\gg = (\gg^{\mathbb
C})^{\sigma}$ of a complex semisimple Lie algebra $\gg^{\mathbb
C}$ is of inner type.  Then, preserving the
notations of Proposition \ref{corAdmissiblepair}, a regular
subalgebra $\gk \subset \gg^{\mathbb C}$ normalized by a maximally
compact Cartan subalgebra $\gh_{\gg}$ of $\gg$ and which is part
of a $\gg$-admissible pair $(\gk, \omega)$ has the form
\begin{equation}\label{k}
\gk =\gh_{0}+ \gg (R^{+})= \gl^{\mathbb{C}} +\mathfrak{a}^{1,0} +
\gg (R^{+}),
\end{equation}
where $R^{+}\subset R$ is a positive root system of $\gg^{\mathbb
C}$ relative to $\gh = (\gh_{\gg})^{\mathbb C}$ and
$\gh_{0}\subset \gh$ satisfies $\gh_{0} + \bar{\gh}_{0} = \gh .$
We now determine all $2$-forms $\omega\in\Lambda^{2}(\gk^{*})$
which, together with $\gk$, form a $\gg$-admissible pair. In the
following theorem we fix root vectors $\{ E_{\alpha}\}_{\alpha\in
R}$ of a Weyl basis of $\gg (R)$ and we denote by
$\omega_{\alpha}\in (\gg^{\mathbb{C}})^{*}$ the dual covectors
defined by
\begin{equation}\label{covectors-def}
\omega_{\alpha}\vert_{\gh}= 0,\quad\omega_{\alpha}(E_{\beta})
=\delta_{\alpha\beta}, \quad\forall\alpha ,\beta\in R.
\end{equation}

As usual $N_{\alpha\beta}$ will denote the structure constants,
defined by $[E_{\alpha}, E_{\beta}]= N_{\alpha\beta}E_{\alpha
+\beta}.$

\begin{thm}\label{main} Let $\gk \subset \gg^{\mathbb{C}}$ be the regular subalgebra
defined by (\ref{k}) and $\omega\in\Lambda^{2}(\gk^{*})$. Then
$(\gk, \omega)$ is a $\gg$-admissible pair if and only if
\begin{equation}\label{forma}
\omega = \widehat{\omega}_{0} +\sum_{\alpha\in
R^{+}}\mu_{\alpha}\alpha\wedge\omega_{\alpha}
+\frac{1}{2}\sum_{\alpha ,\beta\in R^{+}} \mu_{\alpha
+\beta}N_{\alpha\beta}\omega_{\alpha}\wedge\omega_{\beta} ,
\end{equation}
where $\mu_{\alpha}\in\mathbb{C}$, for any $\alpha\in R^{+}$,  and
$\widehat{\omega}_{0}\in \Lambda^{2}(\gh_{0}^{*})$ is any $2$-form
on $\gh_0$ (trivially extended to $\gk$), such that
$\mathrm{Im}\left( \widehat{\omega}_{0} \vert_{\gl}\right)$ is
non-degenerate.
\end{thm}

\begin{proof}
We first show that a $2$-form on $\gk$ is closed if and only if it
is given by (\ref{forma}), for some constants $\mu_{\alpha}$. For
this, we apply Proposition \ref{rho}, with commutative subalgebra
$\gs :=\gh_{0}= \gl^{\mathbb{C}} +\ga^{1,0}$ and ideal $\gp := \gg
(R^{+}).$ Any $\omega\in\Lambda^{2}(\gk^{*})$ is given by
\[
\omega = \widehat{\omega}_{0} + \rho_{1}+ \rho_{2}
\]
where $\widehat{\omega}_{0}\in\Lambda^{2}(\gh_{0}^{*})$,
$\rho_{1}\in\Lambda^{2}(\gp^{*})$ and $\rho_{2}\in
\gh_{0}^{*}\wedge \gp^{*}$ are trivially extended to $\gk .$ Since
$\gh_{0}$ is abelian, $d_{\gh_{0}}\widehat{\omega}_{0}=0$ for any
$\widehat{\omega}_{0}$. From Proposition \ref{rho}, $d_{\gk}\omega
=0$ if and only if $d_{\gp}\rho_{1}=0$ and conditions (\ref{c1})
and (\ref{c2}) are satisfied. Since $\gh_{0}$ is abelian,
condition (\ref{c2}) becomes
\[
\rho_{2}( [H^{\prime}, E_{\alpha}], H) +\rho_{2}( [E_{\alpha}, H],
H^{\prime})=0, \quad \forall H, H^{\prime}\in\gh_{0},\quad
\forall\alpha\in R^{+}
\]
or
\begin{equation}\label{simplificat}
\rho_{2}( E_{\alpha}, H) \alpha (H^{\prime}) = \rho_{2}(
E_{\alpha}, H^{\prime})\alpha (H).
\end{equation}
On the other hand,  for any  $\alpha \in R$, the restriction
$\alpha\vert_{\gh_{0}}$ is not identically zero. (This is an easy
consequence of the equality  $\gh_{0} +\bar{\gh}_{0}=\gh$ and
 the  fact  that  $\sigma (\alpha )
= -\alpha$, when $\gg = (\gg^{\mathbb{C}})^{\sigma}$ is a compact
real form or a real form of inner type and $\gh_{\gg}$ is a
maximally compact Cartan subalgebra of $\gg$). Thus relation
(\ref{simplificat}) is equivalent to
\[
\rho_{2}(H,E_{\alpha}) =\mu_{\alpha}\alpha (H),\quad\forall \alpha
\in R^{+},\quad \forall H\in \gh_{0},
\]
for some constants $\mu_{\alpha}$, or
\begin{equation}\label{rho2}
\rho_{2} =\sum_{\alpha\in R^{+}} \mu_{\alpha}\alpha\wedge
\omega_{\alpha}.
\end{equation}
We proved that (\ref{c2}) is equivalent to (\ref{rho2}). We now
consider (\ref{c1}), which gives
\[
\rho_{2}(H, [E_{\alpha} , E_{\beta}]) =\rho_{1}( [H, E_{\alpha}],
E_{\beta}) +\rho_{1}(E_{\alpha} , [H, E_{\beta}]),\quad\forall
H\in \gh_{0},\quad\forall\alpha ,\beta \in R^{+},
\]
or, using (\ref{rho2}),
\begin{equation}\label{rho3}
N_{\alpha\beta } \mu_{\alpha +\beta} (\alpha +\beta )(H) = (\alpha
+\beta )(H) \rho_{1}(E_{\alpha}, E_{\beta}),\quad\forall H\in
\gh_{0},\quad\forall\alpha ,\beta\in R^{+}.
\end{equation}
A similar argument as above shows that $(\alpha
+\beta)\vert_{\gh_{0}}$ is not identically zero, for any $\alpha ,
\beta \in R^{+}$. Thus (\ref{rho3}) becomes
\begin{equation}\label{rho4}
\rho_{1}= \frac{1}{2} \sum_{\alpha ,\beta\in R^{+}}N_{\alpha\beta
} \mu_{\alpha +\beta}\omega_{\alpha}\wedge\omega_{\beta}.
\end{equation}
It remains to check that $\rho_{1}\in\Lambda^{2}(\gp^{*})$ defined
by (\ref{rho4}) is closed. For this, let $\alpha , \beta , \gamma
\in R^{+}$ be any positive roots. Then
\begin{align*}
(d_{\gp}\rho_{1})(E_{\alpha}, E_{\beta}, E_{\gamma
})&=\rho_{1}(E_{\alpha}, [E_{\beta}, E_{\gamma }])+
\rho_{1}(E_{\gamma}, [E_{\alpha}, E_{\beta }]) +
\rho_{1}(E_{\beta}, [E_{\gamma}, E_{\alpha}])\\
& = c_{\alpha +\beta +\gamma}(N_{\beta\gamma}N_{\alpha , \beta +
\gamma} + N_{\alpha\beta} N_{\gamma , \alpha + \beta} +
N_{\gamma\alpha} N_{\beta , \gamma +\alpha}).
\end{align*}
On the other hand, the Jacobi identity
\[
[E_{\alpha},[ E_{\beta}, E_{\gamma}]] + [E_{\gamma},[ E_{\alpha},
E_{\beta}]]+ [E_{\beta},[E_{\gamma}, E_{\alpha}]]=0
\]
implies that
\[
N_{\beta\gamma}N_{\alpha , \beta + \gamma} + N_{\alpha\beta}
N_{\gamma , \alpha + \beta} + N_{\gamma\alpha} N_{\beta , \gamma
+\alpha}=0.
\]
Therefore $\rho_{1}$ defined by (\ref{rho4}) is closed. We proved
that all closed  $2$-forms on $\gk$ are given by (\ref{forma}).
If, moreover, $\mathrm{Im}\left( \widehat{\omega}_{0}
\vert_{\gl}\right)$ is non-degenerate, then $(\gk ,\omega )$ is a
$\gg$-admissible pair.
\end{proof}

\begin{cor}\label{cor-main}
Let $G$ be a real semisimple Lie group of inner type.

\begin{enumerate}\item[i)]  If there is a regular generalized complex structure on $G$
then the rank of $G$ is even.

\item[ii)]  Assume that the rank of $G$ is even and let $\mathcal J$ be a
regular generalized complex structure on $G$, with associated
admissible pair $(\gk ,\omega )$, where
\begin{equation}\label{111}
\gk = \gl^{\mathbb{C}} +\mathfrak{a}^{1,0} + \gg (R^{+})
\end{equation}
and $\omega\in\Lambda^{2}(\gk^{*})$ are like in Theorem
\ref{main}. Then the type of $\mathcal J$ is given by
\[
\mathrm{type}({\mathcal J}) = \frac{1}{2}\left(
 \mathrm{rank}(G) - \mathrm{dim}(\gl ) + |R|\right)
\]
where $|R|$ is the number of roots from $R$.

\item[iii)] Assume that the rank of $G$ is even. Then, for any $k\in\mathbb{Z}$ such
that
\[
0\leq k\leq \frac{1}{2}\mathrm{rang}(G),
\]
there is a regular generalized complex structure $\mathcal J$ on
$G$ with
\begin{equation}\label{tipul}
\mathrm{type}({\mathcal J}) =\frac{1}{2} | R| + k.
\end{equation}
\end{enumerate}

\end{cor}

\begin{proof}
Let $\mathcal J$ be a regular generalized complex structure on $G$,
with admissible pair $(\gk , \omega )$,
like in Theorem \ref{main}.
Since $\gl$ admits a symplectic form, it has
even dimension. Since the complement $\mathfrak{a}$
of $\gl$ in $\gh_{\gg}$ admits a complex structure, it
has also even dimension. Thus, the rank of $G$ is even. This proves
the first claim. The second claim follows from the definition of the
type and by computing the dimension of $\gk .$ For the third claim, we notice that the generalized complex structure $\mathcal J$
associated to an admissible pair $(\gk ,\omega )$, with $\gk$ given by
(\ref{111}),  satisfies (\ref{tipul}), provided that
the dimension of $\gl$ is equal to
$\mathrm{rank}(G) - 2k .$ Since the rank of $G$ is even, $\gl$ can be chosen with this property.
\end{proof}

For simplicity, the following proposition is stated for compact
forms, but it holds also for any real form of inner type. We preserve
the notations from Theorem \ref{main}.

\begin{prop}\label{B-field}
Let $\mathcal J$ be a regular generalized complex structure on a
compact semisimple Lie group $G$, with associated $\gg$-admissible
pair $(\gk ,\omega )$ defined by (\ref{k}) and (\ref{forma}).
Define a covector $\xi\in \gg^{*}$ by
\[
\xi := -\sum_{\alpha\in R^{+}}\mu_{\alpha}\omega_{\alpha}
+\sum_{\alpha\in R^{+}}\bar{\mu}_{\alpha}\omega_{-\alpha}
\]
and let $B := d\xi$. Then $\mathcal J$ is the $B$-field
transformation of the regular generalized complex structure
$\widehat{\mathcal J}$ whose associated $\gg$-admissible pair is
$(\gk ,\widehat{\omega}_{0}).$ Moreover, $\widehat{\mathcal J}$ is
in normal form if and only if
\[
i_{H}\widehat{\omega}_{0}=0, \quad\forall H\in \mathfrak{a}^{1,0}.
\]
\end{prop}

\begin{proof} The compact real form $\gg$ is
given by
\[
\gg =i\gh ({\mathbb R})+\sum_{\alpha\in R}\mathbb{R}(E_{\alpha} -
E_{-\alpha}) +i \sum_{\alpha\in R}{\mathbb R}(E_{\alpha}
+E_{-\alpha}).
\]
It can be checked that $\xi$ takes real values on $\gg$ and
\begin{equation}\label{xi-d}
B\vert_{\gk}= \sum_{\alpha\in
R^{+}}\mu_{\alpha}\alpha\wedge\omega_{\alpha} +\frac{1}{2}
\sum_{\alpha, \beta\in R^{+}}\mu_{\alpha +\beta}N_{\alpha\beta}
\omega_{\alpha}\wedge\omega_{\beta}.
\end{equation}
From (\ref{forma}) and (\ref{xi-d}), $\mathcal J$ is the $B$-field
transformation of $\widehat{\mathcal J}.$ The second claim is
straightforward.
\end{proof}

\subsection{Invariant generalized complex structures on compact semisimple Lie
groups}\label{igcssc}

We now show that any invariant generalized complex structure
$\mathcal J$ on a compact semisimple Lie group $G$ is regular,
provided that $\mathcal J$ satisfies an additional natural
condition. More precisely, we prove:

\begin{thm}\label{sumarizat}
Let $G$ be a compact semisimple Lie group, with Lie algebra $\gg$,
and $\mathcal J$ an invariant generalized complex structure on $G$
defined by a $\gg$-admissible pair $(\gk ,\omega ).$ Suppose that
$\gl := \gk \cap \gg$ generates a closed subgroup $L$ of $G$. Then
$\mathcal J$ is regular.
\end{thm}

\begin{proof}
Since $G$ is semisimple and compact, $M= G/L$ is reductive and
$\mathcal J$ induces an invariant complex structure $J$ on $M$,
defined by the subalgebra $\gk$ (see Corollary \ref{quot}). By Theorem
\ref{wang}, $\gk$ is regular.
\end{proof}

\section{Invariant generalized complex structures on
semisimple Lie groups of outer type}\label{examples}

In this section we construct a large class of regular generalized
complex structures on semisimple Lie groups of outer type. Let
$\gg =(\gg^{\mathbb{C}})^{\sigma}$ be a real form of outer type of
a complex semisimple Lie algebra $\gg^{\mathbb{C}}$, defined by an
antilinear conjugation
\[
\sigma :\gg^{\mathbb C}\rightarrow\gg^{\mathbb{C}},\quad x\rightarrow
\bar{x}.
\]
In a first stage, we are looking for $\gg$-admissible regular
pairs $(\gk ,\widehat{\omega}_{0})$, such that the $2$-form
$\widehat{\omega}_{0}$ is  the trivial extension of a form defined
on the Cartan part of the regular subalgebra $\gk$ of
$\gg^{\mathbb C}$. Assume, as usual, that
\begin{equation}\label{pentrusimpla}
\gk = \gh_0 + \gg(R_0),
\end{equation}
where $\gh_{0}$ is included in a $\sigma$-invariant Cartan
subalgebra $\gh$ of $\gg^{\mathbb C}$ and $R_{0}$ is a
$\sigma$-positive system  of the root system $R$ of
$\gg^{\mathbb{C}}$ relative to $\gh$ (see Definition
\ref{sigmapara}). Thus, we are looking for pairs $(\gh_{0},
\widehat{\omega}_{0})$, with
$\widehat{\omega}_{0}\in\Lambda^{2}(\gh_{0})$, such that the
following conditions are satisfied:

\begin{enumerate}

\item[i)] $\gh_{0}+ \bar{\gh}_{0} =\gh$;

\item[ii)] the trivial extension of $\widehat{\omega}_{0}$ to $\gk$ is closed;

\item[iii)] $\mathrm{Im}\left(\widehat{\omega}_{0}\vert_{\gl}\right)$ is
non-degenerate, where $\gl := \gh_{0}\cap \gg$ is the real part of
$\gh_{0}$.\\

\end{enumerate}

We first remark that the symmetric part $R_{0}^{\mathrm{sym}}$ of
the $\sigma$-positive system $R_{0}$ is not empty in general (as
it happens when $\gg$ is of inner type). Since $\gk$ is a subalgebra,
\[
[\gg_{\alpha}, \gg_{-\alpha}]\subset \gh_{0},\quad\forall\alpha\in
R_{0}^{\mathrm{sym}}
\]
and a simple computation shows that condition ii) above is
equivalent to
\begin{equation}\label{conditie-closed}
\widehat{{\omega}}_{0} ([E_{\alpha}, E_{-\alpha}],
H)=0,\quad\forall \alpha\in R_{0}^{\mathrm{sym}}, \quad\forall
H\in \gh_{0}.
\end{equation}
Define
\[
\mathcal S :=\mathrm{Span}_{\mathbb{C}}\{ [E_{\alpha} ,
E_{-\alpha}],\quad \alpha\in R_{0}^{\mathrm{sym}}\} =
\mathrm{Span}(R_{0}^{\mathrm{sym}})^{\flat}
\]
where $\flat :\gh^{*} \rightarrow \gh$ is the isomorphism defined by
the Killing form and assume that
\begin{equation}\label{condfin-1}
{\mathcal S}\cap \bar{\mathcal S} =\{ 0 \} .
\end{equation}
\smallskip

Assuming that the additional condition (\ref{condfin-1}) is
satisfied, we now describe a simple construction of pairs
$(\gh_{0}, \widehat{\omega}_{0})$ such that i), ii) and iii) hold.
For this, consider a subspace $\gh_{1}\subset \gh$ such that
$\gh_{1} +\bar{\gh}_{1} =  \gh$ and ${\mathcal S}\subset\gh_{1}.$
From (\ref{condfin-1}), we can chose a complementary subspace
$\gh_{0}$ of $\gh_{1}\cap \bar{\mathcal S}$ in $\gh_{1}$, such
that
${\mathcal S}\subset \gh_{0}.$ One can check that:\\

a) $\gh_{0}\cap \bar{\gh}_{0}$ is transverse to $\mathcal S$;\\

b) $\gh = \gh_{0} +\bar{\gh}_{0}$.\\

The subspace $\gh_{0}$ decomposes into a direct sum
\[
\gh_{0} = (\gh_{0} \cap\bar{\gh}_{0}) \oplus {\mathcal S}\oplus
{\mathcal W}
\]
where $\mathcal W\subset \gh_{0}$ is any complementary subspace of
$(\gh_{0}\cap\bar{\gh}_{0})\oplus {\mathcal S}$. One can chose
$\widehat{\omega}_{0}\in\Lambda^{2}(\gh_{0})$ such that ${\mathcal
S}\subset \mathrm{Ker}(\widehat{\omega}_{0})$ and
$\mathrm{Im}\left(\widehat{\omega}_{0}\right)$ is non-degenerate
on $\gl = (\gh_{0}\cap\bar{\gh}_{0})^{\sigma}.$  Then $(\gk ,
\widehat{\omega}_{0})$ is a $\gg$-admissible pair.\\

Our next aim is to construct other regular $\gg$-admissible pairs
$(\gk , \omega )$, for which $\gg = (\gg^{\mathbb C} )^{\sigma}$
is any real form of outer type of $\gg^{\mathbb C}=
\gsl_{2n}({\mathbb C})$ $(n\geq 2$), $\gso_{2n}(\mathbb{C})$
$(n\geq 3$) and $\ge_{6}$, the regular subalgebra $\gk$ is
normalized by a maximally compact Cartan subalgebra of $\gg$, the
root system $R_{0}$ of $\gk$ is a $\sigma$-positive system and the
$2$-form $\omega$ has a non-zero restriction to the root part
$\gg(R_0)$ of $\gk .$ The following Proposition simplifies considerably
this task. It shows that we can chose (arbitrarily), for each of
the above complex Lie algebras, a real form of outer type and find
regular pairs with respect to these real forms. They will provide
invariant generalized complex structures on all real forms of
outer type of $SL_{2n}({\mathbb C})$ $(n\geq 2)$,
$SO_{2n}({\mathbb C})$ $(n\geq 3$) and $E_{6}.$

\begin{prop}\label{remarca}
Let $\gg^{\mathbb C}$ be a complex simple Lie algebra, $\gg =
(\gg^{\mathbb C})^{\sigma}$ a real form of outer type of $\gg^{\mathbb C}$ and
$(\gk , \omega )$ a regular $\gg$-admissible pair, with
\begin{equation}\label{kappa-fin}
\gk = \gh_{\gk} + \gg (R_{0})
\end{equation}
normalized by a maximally compact Cartan subalgebra $\gh_{\gg}$ of
$\gg$ and $R_{0}$ a $\sigma$-positive system.
Then any other real form of outer type of $\gg^{\mathbb{C}}$ is conjugated to a real form
$\gg^{\prime}$ such that $(\gk ,\omega )$ is
$\gg^{\prime}$-admissible.
\end{prop}

\begin{proof}
To prove the claim, we
recall the
formalism of Vogan diagrams
 \cite[p.\,344]{knapp}, which
describes the real forms of a given complex simple Lie algebra $\gg^{\mathbb{C}} .$
An abstract Vogan diagram is
an abstract Dynkin diagram, which represents a system of simple roots $\Pi$
of $\gg^{\mathbb{C}}$ relative to a Cartan subalgebra $\gh$,
together with two pieces of additional
information: some arrows between vertices, which indicate the action of an involutive symmetry $s$
of the Dynkin diagram (in the case of outer real forms)  and a subset of
painted nodes (in the fix point set of $s$), which indicates the non-compact simple roots.
The symmetry $s$ defines an involution $\theta$ of the Cartan subalgebra $\gh$,
which can be canonically extended to an involutive automorphism $\theta$ of $\gg^{\mathbb{C}}$,
by the conditions: for any $\alpha \in \Pi$,
\begin{equation}
\theta (E_{\alpha })=\begin{cases}
E_{\alpha^{\prime}},\quad \mathrm{if}\ s(\alpha )= \alpha^{\prime}\neq \alpha\\
E_{\alpha},\quad \text{if}\ \alpha = s(\alpha )\  \mathrm{is\ not\ a\ painted\ root}\\
- E_{\alpha},\quad \text{if}\ \alpha = s(\alpha )\  \mathrm{is\ a\ painted\ root},
\end{cases}
\end{equation}
where $\{ E_{\alpha}\}$ are root vectors of a Weyl basis.

The composition $\sigma := \theta\circ \tau$, where $\tau$ is the compact involution
commuting with $\theta$,
defines the real form $\gg = (\gg^{\mathbb{C}})^{\sigma}$ associated to the Vogan diagram.
Moreover, the real subalgebra
\begin{equation}\label{cartan-special}
\gh_{\gg} = \gh^{+}\oplus \gh^{-}=  \{ x\in i \gh ({\mathbb R}), \quad \theta (x) = x\}
\oplus \{ x\in \gh ({\mathbb R}), \quad \theta (x) = -x \}
\end{equation}
is a maximally compact ($\theta$-invariant) Cartan subalgebra of $\gg$.

Consider now a regular $\gg$-admissible pair
$(\gk ,\omega )$, like in (\ref{kappa-fin}),
 where $\gg$ is a real form of outer type of a complex simple Lie algebra $\gg^{\mathbb{C}}$.
The defining conditions for  $(\gk , \omega )$, namely
\[
R_{0} \cup \sigma ( R_{0}) = R,\quad R_{0} \cap \sigma (R_{0}) =
\emptyset ,\quad \gh_{\gk}+ \bar{\gh}_{\gk}= \gh ,\quad
d_{\gk}\omega = 0
\]
and $\mathrm{Im}(\omega\vert_{(\gh_{\gk}\cap \bar{\gh}_{\gk})^{\sigma}})$
non-degenerate, depend only on the symmetry $s$ of the associated Vogan
diagram.  Remark now that the symmetry $s$ is unique,
for $\gg^{\mathbb{C}} \neq D_{4}$ and for $D_{4}$ there are three symmetries related by
an outer automorphism of $D_{4}.$ This proves our claim.

\end{proof}

\begin{rem}
In the statement of Proposition \ref{remarca} it is essential that
$R_{0}$ is a $\sigma$-positive system.  When $R_{0}$ is $\sigma$-parabolic
but not necessarily $\sigma$-positive,
$(\gk\cap \bar{\gk})^{\sigma}$ does not reduce in general to the Cartan part
$(\gh_{\gk}\cap \bar{\gh}_{\gk})^{\sigma}$
and it may depend also on which nodes from the Vogan diagram are painted; therefore,
the same is true for the condition  $\mathrm{Im}(\omega\vert_{(\gk\cap \bar{\gk})^{\sigma}})$
from the definition of $\gg$-admissible pairs.
\end{rem}

Motivated by Proposition \ref{remarca}, in the remaining part of this
section we construct regular generalized complex structures (and
admissible pairs) on $G= SL_{n}({\mathbb H})$ ($n\geq 2$),
$SO_{2n-1,1}$ ($n\geq 3$) and two real forms of outer type of
$E_{6}.$

\subsection{Generalized complex structures on $SL_{n}({\mathbb
H})$}\label{slnh}

{\bf a) Description of  the antiinvolution $\sigma$ which defines
$\gsl_n({\mathbb{H}})$}\\

Let $V =\mathbb{C}^{2n}$ be a  complex vector space of dimension
$2n\geq 4$, with standard basis $\{ e_{1}, \cdots ,
e_{n},e_{1^{\prime}}, \cdots , e_{n^{\prime}}\}$ and
$\gsl_{2n}({\mathbb C})$ the Lie algebra of traceless
endomorphisms of $V$. We  denote by
\[
E_{ij} = e_i \otimes e_j^*, \, E_{i^{\prime}j^{\prime}}= e_{i^{\prime}}
\otimes e^*_{j^{\prime}},\,
 E_{i^{\prime}j} = e_{i^{\prime}} \otimes e^*_{j}    ,\,
 E_{ij^{\prime}}= e_{i} \otimes e^*_{j^{\prime}}
 \]
the  associated  basis of $\mathfrak{gl}(V)$ and we  choose  a
Cartan subalgebra
\[
\gh = \{ H = \sum_{i=1}^{n} x_i E_{ii} + \sum_{j'=1}^n x_{j'}
E_{j'j'},\,\, \sum_{i=1}^{n} x_i + \sum_{j^{\prime}=1}^{n} x_{j'}
=0\}
\]
which consists of traceless  diagonal matrices. Denote by
$\epsilon_i, \epsilon_{j'}$ the linear forms on $\gh$ defined by
$$\epsilon_i(H) = x_i,\,\,\,\,\,  \epsilon_{j'}(H) = x_{j'}.$$
The roots of $\gsl(V)$ are
\[
R:= \{\epsilon_{ij}:=\epsilon_{i}-\epsilon_{j},
\epsilon_{i^{\prime}j^{\prime}}:=\epsilon_{i^{\prime}}-\epsilon_{j^{\prime}},
\epsilon_{i^{\prime}j}:= \epsilon_{i^{\prime}}-\epsilon_{j},
\epsilon_{ij^{\prime}}:= \epsilon_{i}-\epsilon_{j^{\prime}}\}.
\]
The Lie algebra ${\mathfrak{sl}}_n(\mathbb{H})$ is a real form of
outer type of $\gsl_{2n}({\mathbb C})$,  defined by the antilinear
involution
\[
\sigma (A) = -J\bar{A}J,\quad\forall A\in
{\mathfrak{sl}}_{2n}({\mathbb C}),
\]
where $J$ is the  matrix
\[
J = \left(\begin{tabular}{cc} $0$ & $\Id$\\
$-\Id$ & $0$
\end{tabular}\right).
\]
The antilinear involution $\sigma$ acts on roots, transforming
unprime indices  into  prime indices  and vice versa, i.e.
\[
 \sigma (\epsilon_{ij})= \epsilon_{i^{\prime}j^{\prime}},\quad
\sigma (\epsilon_{i^{\prime}j^{\prime}})= \epsilon_{ij},\quad
\sigma (\epsilon_{ij^{\prime}})= \epsilon_{i^{\prime}j},
\quad\sigma (\epsilon_{i^{\prime}j})= \epsilon_{ij^{\prime}}.
\]
Alternatively, one may also define $\gsl_{n}({\mathbb H})$ as the
unique real form of outer type of $\gsl_{2n}({\mathbb C})$, whose
Vogan diagram has no painted node. As described above,
$\gsl_{n}({\mathbb H})$ can be obtained as in the proof of Theorem
6.88 of \cite{knapp} (see also Proposition \ref{remarca}), by
considering the Weyl basis (\ref{vectori}) (see below) and assigning to the
nodes of $A_{2n-1}$ the simple roots $\epsilon_{12^{\prime}},
\epsilon_{2^{\prime}3}, \epsilon_{34^{\prime}},\cdots  ,
\epsilon_{3^{\prime}2}, \epsilon_{21^{\prime}}$. The unique simple
root fixed by the non-trivial automorphism of $A_{2n-1}$ (given by
the horizontal reversal) is $\epsilon_{n^{\prime}n}$ ($n$-even) or
$\epsilon_{nn^{\prime}}$ ($n$-odd) and is a white root. The Cartan
subalgebra
\[
\gh^{\sigma} = \{ H = \sum_{i=1}^{n}x_{i} E_{ii} +\sum_{j=1}^{n}
\overline{x}_{j}E_{j^{\prime}j^{\prime}},\quad \sum_{i=1}^{n}(
x_{i} +\bar{x}_{i})=0\}
\]
is a maximally compact Cartan subalgebra of $\gsl_{n}({\mathbb
H})$. Any $\alpha\in R$ for which $\sigma (\alpha ) = -\alpha$ is
of the form $\epsilon_{ii^{\prime}}$ or $\epsilon_{i^{\prime}i}$
and is compact, since $\sigma (E_{\epsilon_{ii^{\prime}}}) = -
E_{\epsilon_{i^{\prime}i}}$, where $E_{\epsilon_{ij}}$ are root
vectors of the Weyl basis (\ref{vectori}). It follows that
$\gh^{\sigma}$ is also a maximally non-compact Cartan subalgebra
of $\gsl_{n}({\mathbb H})$  (see \cite[p.\,335]{knapp}).

\vspace{10pt}

{\bf b) $\sigma$-positive  systems of  the Lie algebra
$\gsl_{2n}({\mathbb C})$}

\begin{prop}\label{clasif1}
Any  $\sigma$-positive  root system  $R_0$ of the root system $R$
is  equivalent to one of the following systems:

a) $\{ \epsilon_i - \epsilon_j,\, \epsilon_{i} -
\epsilon_{j'},\,\, i,j = 1,2,\cdots  ,n\}$;

b) $\{ \epsilon_i - \epsilon_j,\, \epsilon_{i'} - \epsilon_j,\,\,
i,j = 1,2, \cdots  ,n \}.$

\end{prop}

\begin{proof} Denote  by $P = \{ \epsilon_{ij} \}$  the subsystem of $R$  which is the root system of type
$A_{n-1}$, hence  indecomposable. But $P = (P \cap R_0) \cup (P
\cap \sigma(R_0))$ is  a decomposition of $P$ in a disjoint union
of two closed  subsystems. Hence, one of these two parts is empty.
Without loss of  generality, we may assume that $P \subset R_0$.
Assume that $\epsilon_{i_0 j'_0} \in R_0$. Then $\epsilon_{ij_0'}
= \epsilon_{ii_{0}} +\epsilon_{i_{0}j_{0}^{\prime}} \in R_0 \,\,
\forall i$, because $R_{0}$ is closed, and $\epsilon_{i'j_0}
=\sigma (\epsilon_{ij^{\prime}_{0}})\in \sigma(R_0 )$ $\forall
i^{\prime}$. It follows that $\epsilon_{i'j} \in \sigma(R_0)
\,\,\,\, \forall j$ (if this were not true, then we could find
$j_{1}$ such that $\epsilon_{i^{\prime}j_{1}} \in R_{0}$; but
then, since $\epsilon_{j_{1}j_{0}}\in R_{0}$ and $R_{0}$ is
closed, also $\epsilon_{i^{\prime}j_{0}}\in R_{0}$ which is
impossible, since $R_{0}$ and $\sigma (R_{0})$ are disjoint). Thus
$\epsilon_{ij'} =\sigma (\epsilon_{i^{\prime}j})\in R_0 \,\,\,\,
\forall i,j^{\prime}$ and we get the system {\it a).} If
$\epsilon_{ij'} \notin R_0 \,\,\, \forall i,j^{\prime}$, we get
the system {\it b).}
 \end{proof}

\vspace{10pt}

{\bf c) $\gsl_{n}({\mathbb H})$-admissible pairs}\\

Now we describe $\gsl_{n}({\mathbb H})$-admissible pairs $(\gk,
\omega)$, where the subalgebra $\gk\subset \gsl_{2n}({\mathbb C})$
is regular with root system the $\sigma$-positive system $R_0$ of
type {\it a)} from Proposition \ref{clasif1}. The case {\it b)} is
similar. The  subalgebra $\gk$ can be written as
\begin{equation}\label{hh}
\gk =  \gh_0 + \gg(R_0)= \gh_{0}+ \sum_{i,j}
\gg_{\epsilon_{ij}}+\sum_{i, j^{\prime}}
\gg_{\epsilon_{ij^{\prime}}}\subset \gsl_{2n}({\mathbb C}),
\end{equation}
where $\gh_{0}\subset \gh$, with $\gh_{0} + \bar{\gh}_{0}= \gh$.
The vectors
\begin{equation}\label{vectori}
 E_{\epsilon_{ij}}=\frac{1}{\sqrt{2n}}E_{ij}, \quad
E_{\epsilon_{i^{\prime}j^{\prime}}}=\frac{1}{\sqrt{2n}}E_{i^{\prime}j^{\prime}},\quad
E_{\epsilon_{i^{\prime}j}}=\frac{1}{\sqrt{2n}}E_{i^{\prime}j},\quad
E_{\epsilon_{ij^{\prime}}}=\frac{1}{\sqrt{2n}}E_{ij^{\prime}}
\end{equation}
are root vectors of a Weyl basis and the associated structure
constants are given by
\[
N_{\epsilon_{ij},\epsilon_{js}}= - N_{\epsilon_{ji},
\epsilon_{sj}}= \frac{1}{\sqrt{2n}},\quad\forall i\neq j\neq s
\]
and their prime analogues (obtained by replacing any of the $i,j,
s$ by its prime analogue). The symmetric part
$R_{0}^{\mathrm{sym}}$ of $R_{0}$ is given by
\begin{equation}\label{lm}
R_{0}^{\mathrm{sym}}= \{\epsilon_{ij},\ i,j=1,2,\cdots , n\} .
\end{equation}
Since $\gk$ is a subalgebra, (\ref{lm}) together with
\[
[E_{\epsilon_{ij}}, E_{\epsilon_{ji}}] = \frac{1}{2n} (E_{ii} -
E_{jj})
\]
imply that $E_{ii} - E_{jj} \in \gh_{0}$, for any $i,j.$ Note that
\[
{\mathcal S}: =\mathrm{Span}_{\mathbb C}\{ [E_{\epsilon_{ij}},
E_{\epsilon_{ji}}],i,j=1,2,\cdots , n\} = \mathrm{Span}_{\mathbb C}\{ E_{ii} -
E_{jj}, i,j=1,2,\cdots , n\}
\]
is transversal to
\[
\bar{\mathcal S} =  \mathrm{Span}_{\mathbb C}\{
E_{i^{\prime}i^{\prime}} - E_{j^{\prime}j^{\prime}},i^{\prime},
j^{\prime}=1,2,\cdots , n\} .
\]
Thus condition (\ref{condfin-1}) is satisfied. The following
theorem describes $\gsl_{n}({\mathbb H})$-admissible pairs $(\gk,
\omega)$ and hence also invariant generalized complex structures
on the group $SL_n(\mathbb{H})$. Below the covectors
$\omega_{\epsilon_{ij}},\omega_{\epsilon_{i^{\prime}j}},
\omega_{\epsilon_{ij^{\prime}}},
\omega_{\epsilon_{i^{\prime}j^{\prime}}} $ are dual to the root
vectors (\ref{vectori}) and annihilate $\gh$. We assume that
$n\geq 3$ (when $n=2$ the theorem still holds, under the
additional assumption that $\epsilon_{i} +\epsilon_{r}
-\epsilon_{j^{\prime}} - \epsilon_{s^{\prime}}$ is not identically
zero on $\gh_{0}$, for any $i, r, j^{\prime}, s^{\prime}$, with
$(i,j^{\prime}) \neq (r,s^{\prime})$).

\begin{thm}  Any closed  $2$-form $\omega$ on the Lie algebra $\gk$ defined
in (\ref{hh}) is given by
\begin{align*}
\omega &=\widehat{\omega}_{0}+\sum_{i\neq j}
\lambda_{(ij)}\epsilon_{ij}\wedge
\omega_{\epsilon_{ij}}+\frac{1}{\sqrt{2n}} \sum_{i\neq j\neq k}
\lambda_{(ik)} \omega_{\epsilon_{ij}}\wedge\omega_{\epsilon_{jk}}
+
\sum_{i\neq j}\eta_{(ij)}\omega_{\epsilon_{ij}}\wedge\omega_{\epsilon_{ji}}\\
&+ \sum_{k, j^{\prime}} \lambda_{(kj^{\prime})}\left(
\sqrt{2n}\epsilon_{kj^{\prime}}\wedge
\omega_{\epsilon_{kj^{\prime}}} +\sum_{i\neq k
}\omega_{\epsilon_{ki}}\wedge
\omega_{\epsilon_{ij^{\prime}}}\right)
\end{align*}
where $\widehat{\omega}_{0}\in\Lambda^{2}(\gh_{0}^{*})$ is such
that
\begin{equation}\label{conditie-omega}
\widehat{\omega}_{0}(E_{ii}- E_{jj},\cdot )=0,\quad\forall i, j
\end{equation}
and
$\lambda_{(ij)},\lambda_{(ij^{\prime})},\eta_{(ij)}\in\mathbb{C}$.
The pair $(\gk, \omega)$ is $\gsl_{n}({\mathbb H})$-admissible,
hence it defines a regular generalized complex structure on
$SL_{n}({\mathbb H})$, if and only if the real 2-form
$\mathrm{Im}\left(\widehat{\omega}_{0}\right)$ is non-degenerate
on $\gh_{0}\cap\gsl_{n}({\mathbb H})$.
\end{thm}

\begin{proof}
To  determine all closed $2$-forms $\omega$ on  $\gk$, we apply
Proposition \ref{rho}, with subalgebra $\gs := \gh_{0} + \gg (\{
\epsilon_{ij}\} )$ and ideal $\gp := \gg (\{
\epsilon_{ij^{\prime}}\} ).$ In terms of the decomposition
\[
\omega = \rho_{0} +\rho_{1}+\rho_{2},
\]
the form $\omega$ is closed if and only if its $\gs$-part
$\rho_0\in\Lambda^{2}(\gs^{*})$ and $\gp$-part
$\rho_1\in\Lambda^{2}(\gp^{*})$ are closed and the mixed part
$\rho_2\in\gs^{*}\wedge \gp^{*}$ satisfies (\ref{c1}) and
(\ref{c2}). Using (\ref{c2}), we determine $\rho_{2}$, as follows.
The condition (\ref{c2}) is equivalent to the following three conditions:\\
i) for any $H\in \gh_{0}$  and root vectors $E_{\epsilon_{kp}}$
and $E_{\epsilon_{ij^{\prime}}}$,
\begin{equation}\label{ecuatia1}
\rho_{2}( [H, E_{\epsilon_{kp}}], E_{\epsilon_{ij^{\prime}}})
+\rho_{2}( [E_{\epsilon_{kp}}, E_{\epsilon_{ij^{\prime}}}], H) +
\rho_{2}([E_{\epsilon_{ij^{\prime}}}, H], E_{\epsilon_{kp}})=0;
\end{equation}
ii) for any $H, H^{\prime}\in \gh_{0}$ and root vector
$E_{\epsilon_{ij^{\prime}}}$,
\begin{equation}\label{ecuatia2}
\rho_{2} ([H^{\prime}, E_{\epsilon_{ij^{\prime}}}], H) +\rho_{2}(
[E_{\epsilon_{ij^{\prime}}}, H], H^{\prime})=0;
\end{equation}
iii) for any root vectors $E_{\epsilon_{ks}}$,
$E_{\epsilon_{pq}}$, $E_{\epsilon_{ij^{\prime}}}$,
\begin{equation}\label{ecuatia3}
\rho_{2}([E_{\epsilon_{pq}}, E_{\epsilon_{ks}}],
E_{\epsilon_{ij^{\prime}}}) + \rho_{2}( [E_{\epsilon_{ks}},
E_{\epsilon_{ij^{\prime}}}], E_{\epsilon_{pq}}) + \rho_{2}(
[E_{\epsilon_{ij^{\prime}}}, E_{\epsilon_{pq}}],
E_{\epsilon_{ks}})=0.
\end{equation}
It can be checked that (\ref{ecuatia1})
is equivalent to
\begin{equation}\label{sub}
\rho_{2}= \sum_{k,j^{\prime}} \lambda_{(kj^{\prime})}\left(
\sqrt{2n} \epsilon_{kj^{\prime}}\wedge
\omega_{\epsilon_{kj^{\prime}}} +\sum_{i\neq
k}\omega_{\epsilon_{ki}}\wedge\omega_{\epsilon_{ij^{\prime}}}
\right) ,
\end{equation}
for some constants $\lambda_{(kj^{\prime})}\in\mathbb{C}.$
Moreover, with $\rho_{2}$  defined by (\ref{sub}), the relations
(\ref{ecuatia2}) and (\ref{ecuatia3}) are satisfied. Thus
condition (\ref{c2}) is equivalent to (\ref{sub}). We now show,
using condition (\ref{c1}), that $\rho_{1}=0.$ Since $\gp$ is
abelian, condition (\ref{c1}) implies that
\[
\rho_{1}([H, E_{\epsilon_{ij^{\prime}}}],
E_{\epsilon_{rs^{\prime}}}) +\rho_{1}( E_{\epsilon_{ij^{\prime}}},
[H,E_{\epsilon_{rs^{\prime}}}])=0,
\]
or
\[
(\epsilon_{i}- \epsilon_{j^{\prime}} +\epsilon_{r}
-\epsilon_{s^{\prime}})(H) \rho_{1}( E_{\epsilon_{ij^{\prime}}},
E_{\epsilon_{rs^{\prime}}})=0.
\]
Let $H:= E_{ii} - E_{jj}$ with $j\neq i,r$ (we assumed that $n\geq
3$). Since $\epsilon_{i}- \epsilon_{j^{\prime}} +\epsilon_{r}
-\epsilon_{s^{\prime}}$ takes non-zero value on $H$,
\begin{equation}\label{red}
\rho_{1}=0.
\end{equation}
Finally, it remains to determine the closed form $\rho_{0}$ on
$\gs .$  The Lie algebra  $\gs$ is reductive with semisimple part
generated by $\gg (\{ \epsilon_{ij}\})$ and center which is  the
annihilator of $\{ \epsilon_{ij}\}$ in $\gh_{0}.$ Using
Proposition \ref{rho} and the fact that any closed $2$-form on a
semisimple Lie algebra is exact, it can be shown that
\begin{equation}\label{red1}
\rho_{0} =\widehat{\omega}_{0}+\sum_{i\neq j}
\lambda_{(ij)}\epsilon_{ij}\wedge
\omega_{\epsilon_{ij}}+\frac{1}{\sqrt{2n}} \sum_{i\neq j\neq k}
\lambda_{(ik)} \omega_{\epsilon_{ij}}\wedge\omega_{\epsilon_{jk}}
+ \sum_{i\neq
j}\eta_{(ij)}\omega_{\epsilon_{ij}}\wedge\omega_{\epsilon_{ji}}
\end{equation}
where  $\lambda_{(ij)}, \eta_{(ij)}\in\mathbb{C}$  are arbitrary
constants and the $2$-form
$\widehat{\omega}_{0}\in\Lambda^{2}(\gh_{0}^{*})$ satisfies
(\ref{conditie-omega}). Combining (\ref{sub}), (\ref{red}) and
(\ref{red1}) we conclude that all closed $2$-forms on $\gk$ are
described in the theorem. The last claim is straightforward.

\end{proof}

\subsection{Generalized complex structures on
$SO_{2n-1,1}$}\label{so}

{\bf a) Description of the antiinvolution $\sigma$ which defines
$\gso_{2n}({\mathbb C})$}\\

Let $(V, (\cdot ,\cdot ) )$ be a complex Euclidean vector space of
dimension $2n\geq 6$ and $\gso (V) \simeq \gso_{2n}({\mathbb{C}})$
the associated complex orthogonal Lie algebra. We identify
$\gso(V)$ with $\Lambda^2V$ using the scalar product $(\cdot
,\cdot )$ and we choose a basis $e_i, e_{-i},\, i = 1, \cdots ,n$
of $V$ with the only non-zero scalar products $(e_i, e_{-i}) =1$.
The diagonal Cartan subalgebra $\gh \subset \gso(V)$ has   a basis
\[
\{ H_{i}:= e_i \wedge e_{-i}, \quad i=1,2,\cdots , n\}.
\]
We denote by  $\{\epsilon_{i}\}$  the dual basis of $\gh^{*}$.
Then the root system of $\gso(V)$ relative to $\gh$ is given by
\[
R:= \{ \pm\epsilon_{i}\pm \epsilon_{j},\quad i, j=1,2,\cdots
,n,i\neq j\}
\]
 and  the root vectors of a Weyl basis  are
\begin{align*}
E_{\epsilon_{i} + \epsilon_{j}}&:=
\frac{1}{\sqrt{2(n-1)}}(e_i \wedge e_{j}),\quad i < j\\
E_{-\epsilon_{i}-\epsilon_{j}}&:= -\frac{1}{\sqrt{2(n-1)}}(e_{-i}
\wedge e_{-j}),\quad i<j\\
E_{\epsilon_{i} - \epsilon_{j}}&:=
\frac{1}{\sqrt{2(n-1)}}(e_i \wedge e_{-j}),\quad i\neq j.\\
\end{align*}
The associated structure constants are given by
\begin{align*}
N_{\epsilon_{i}+\epsilon_{j},\epsilon_{k}-\epsilon_{j}} &= -\frac{1}{\sqrt{2(n-1)}}\gamma_{ij}\gamma_{ik}\\
N_{-(\epsilon_{i}+\epsilon_{j}),\epsilon_{l}+\epsilon_{j}}&= \frac{1}{\sqrt{2(n-1)}}\gamma_{ij}\gamma_{jl}\\
N_{-(\epsilon_{i}+\epsilon_{j}),\epsilon_{j}-\epsilon_{k}}& = \frac{1}{\sqrt{2(n-1)}}\gamma_{ij}\gamma_{ik}\\
N_{\epsilon_{i} -\epsilon_{j}, \epsilon_{j} -\epsilon_{k}}&=
\frac{1}{\sqrt{2(n-1)}},
\end{align*}
where $\gamma_{ij}=1$ if $i<j$ and $-1$ if $i>j.$

Consider the antilinear involution $\sigma$  of $V$ defined by
\[
\sigma (e_{\pm i}) = e_{\mp i}, \quad
1\leq i<n, \quad
\sigma (e_{\pm n}) = e_{\pm n}.
\]
It induces an antilinear involution $\sigma$ on $\gso (V)$ whose
associated real form is the Lorentzian Lie algebra
$\gso_{2n-1,1}$. The map $\sigma$ preserves the Cartan subalgebra
$\gh$ and  it acts on the weights $\epsilon_i$ as follows:
\[
\sigma (\epsilon_i) = - \epsilon_i ,\,\,  i <n,\,\,   \sigma(\epsilon_n) = \epsilon_n.
\]
Alternatively, one may also define $\gso_{2n-1,1}$ as the unique
real form of outer type of $\gso_{2n}({\mathbb C})$, whose Vogan
diagram has no painted node and the automorphism is given by
interchanging the two ends of $D_{n}.$  As described above,
$\gso_{2n-1,1}$ can be obtained as in the proof of Theorem 6.88 of
\cite{knapp} (see also Proposition \ref{remarca}), by considering the
root vectors of the Weyl basis defined above and assigning to the
nodes of $D_{n}$ the simple roots $ \epsilon_{1}-\epsilon_{2},
\cdots , \epsilon_{n-2} -\epsilon_{n-1},
\epsilon_{n-1}-\epsilon_{n}, \epsilon_{n-1}+\epsilon_{n}.$ The
roots $\epsilon_{i}-\epsilon_{i+1}$, $1\leq i\leq n-2$ are
compact. The automorphism $\theta$ of $D_{n}$ is given by
\[
\theta (\epsilon_{i}-\epsilon_{i+1}) =\epsilon_{i}-
\epsilon_{i+1}, \quad 1\leq i\leq n-2;\quad \theta
(\epsilon_{n-1}- \epsilon_{n}) = \epsilon_{n-1}+\epsilon_{n}
\]
and
\[
\gh^{\sigma} = \{ H= \sum_{k=1}^{n-1} i x_{k} H_{k} + x_{n}H_{n},
\quad x_{k}\in\mathbb{R}\}
\]
is a maximally compact Cartan subalgebra of $\gso_{2n-1,1}.$ It is
also maximally non-compact (easy check).

\vspace{10pt}

{\bf b) $\sigma$-positive systems of the Lie algebra
$\gso_{2n}({\mathbb C})$}\\

We denote by $R' \subset R$ the root system of the subalgebra
${\gso}_{2n-2}({\mathbb C})\subset {\gso}(V)$ which preserves the
vectors $e_{\pm n}$. Then
\[
 \sigma|_{R'} = -1 \,\, \mathrm{and} \,\, \sigma(\epsilon_{n-1}-
\epsilon_{n}) = - ( \epsilon_{n-1}+\epsilon_{n} ).
\]

\begin{prop}\label{clasif2}
Any  $\sigma$-positive system  $R_0 \subset R$  is equivalent   to
one of the systems:\\
 a) $R_0 = R^+ := \{ \epsilon_{i} \pm \epsilon_{j}, 1\leq i<j\leq n\}$;\\
b) $R_0 = \left( R^{+}\setminus \{ \epsilon_{n-1}
 +\epsilon_{n}\}\right) \cup
 \{ \epsilon_{n}-\epsilon_{n-1}\} ;$\\
c) $R_0 = \left( R^{+}\setminus \{ \epsilon_{n-1}
 -\epsilon_{n}\}\right) \cup
 \{ -(\epsilon_{n-1}+\epsilon_{n})\} .$

\end{prop}

\begin{proof} Since $\sigma =-1$ on $R'$,  we  may assume that $R_0$
contains  the  positive root  system $ R'_+= \{\epsilon_i \pm
\epsilon_j,\, i<j <n\}$ of $R'$. Let $\alpha_{n-1}:=\epsilon_{n-1}
-\epsilon_{n}$ and
 $\alpha_{n}:= \epsilon_{n-1}+\epsilon_{n}.$
Since $\sigma$ interchanges the
pairs $(\alpha_{n-1}, \alpha_n)$  and $(- \alpha_n, -
\alpha_{n-1})$, the system  $R_{0}$ must contain one of the pairs
 \[
 (\alpha_{n-1}, \alpha_n),\,(-\alpha_{n-1}, -\alpha_n),
(\alpha_{n-1}, -\alpha_{n-1}),\, (\alpha_{n}, -\alpha_n) .
\]
One can easily check that the composition
$s_{\epsilon_{n-1}-\epsilon_{n}}\circ
s_{-(\epsilon_{n-1}+\epsilon_{n})}$ of reflections exchanges the
first two pairs of roots. Thus the closed systems spanned by
$R'_+$ and any of the first two pairs are equivalent and give the
positive root system $R^+$. The closed systems spanned by
$R_{+}^{\prime}$ and the remaining two pairs give the systems {\it
b)} and {\it c)}.
\end{proof}

\vspace{10pt}

{\bf c) $SO_{2n-1,1}$-admissible pairs.}\\

Since  the $\sigma$-positive  system $R_0$ of type {\it a)} found
in Proposition \ref{clasif2} is a system of positive roots, the
corresponding  admissible  pairs  can be  described using the
method of Theorem \ref{main}. Now, we describe the
$\gso_{2n-1,1}$-admissible pairs $(\gk, \omega)$,  where
$\gk\subset \gso_{2n}({\mathbb C})$ is a regular subalgebra with
the root system $R_0$ of type {\it c)} (the case {\it b)} can be
treated similarly). Let
\begin{equation}\label{gko}
\gk := \gh_{0} + \gg( R_{0})\subset \gso_{2n}(\mathbb{C}),
\end{equation}
where  $\gh_{0}$ is included in the diagonal  Cartan subalgebra
$\gh$ of $\gso_{2n}(\mathbb{C})$ and $R_{0}$ is given by
Proposition \ref{clasif2} {\it c)}:
\[
R_0 = \left( R^{+}\setminus \{ \epsilon_{n-1}
 -\epsilon_{n}\}\right) \cup
 \{ -(\epsilon_{n-1}+\epsilon_{n})\} ,\quad R^{+}= \{
 \epsilon_{i}\pm \epsilon_{j}, 1\leq i<j\leq n\} .
\]
Since
\begin{equation}\label{ry}
R_{0}^{\mathrm{sym}}=\{ \pm (\epsilon_{n-1}+\epsilon_{n})\}
\end{equation}
and $\sigma (\epsilon_{n-1}+\epsilon_{n})= - \epsilon_{n-1} +\epsilon_{n}$,
condition (\ref{condfin-1}) is satisfied.
Since $\gk$ is a subalgebra, (\ref{ry})
together with
\[
[E_{\epsilon_{n-1}+\epsilon_{n}}, E_{-(\epsilon_{n-1}+\epsilon_{n})}]
=\frac{1}{2(n-1)}(H_{n-1}+H_{n})
\]
imply that $H_{n-1}+H_{n}$ belongs to $\gh_{0}.$ Thus,
\[
\gh_{0} = \left(
\mathrm{Ker}(\epsilon_{n-1}+\epsilon_{n})\cap\gh_{0}\right) \oplus
\mathrm{Span}(H_{n-1}+H_{n}).
\]
Let
\[
R_{0}^{\prime}: = R^{+}\setminus \{ \epsilon_{n-1}\pm
\epsilon_{n}\} = R_{0}\setminus \{ \pm
(\epsilon_{n-1}+\epsilon_{n}) \} .
\]
For simplicity, we assume that for any $\alpha\in R^{\prime}_{0}$,
\begin{equation}\label{revizuit}
\alpha\vert_{\mathrm{Ker}(\epsilon_{n-1}+\epsilon_{n})\cap
\gh_{0}}\not\equiv 0
\end{equation}
Remark that the above condition is automatically true, when
$\alpha = \epsilon_{i} \pm \epsilon_{j}$, $i< j\leq n-1$ (recall
that $\sigma\vert_{R^{\prime}}= -1$ and $\gh_{0}+ \bar{\gh}_{0}=
\gh$). It also holds for any $\alpha\in R^{\prime}_{0}$, when
$\gh_{0} = \gh$.

\begin{thm} Any  closed $2$-form $\omega$ on the Lie algebra  $\gk$
defined in (\ref{gko})  is given by

\begin{align*}
\omega &= (\epsilon_{n-1}+ \epsilon_{n})\wedge \left( a
\omega_{\epsilon_{n-1}+ \epsilon_{n}} + b
\omega_{-(\epsilon_{n-1}+ \epsilon_{n})}\right)+
c\omega_{\epsilon_{n-1}+\epsilon_{n}}\wedge
\omega_{-(\epsilon_{n-1}+\epsilon_{n})}\\
&+\widehat{\omega}_{0} + \sum_{\alpha\in
R_{0}^{\prime}}c_{\alpha}\alpha\wedge \omega_{\alpha} +\frac{1}{2}
\sum_{\alpha\in R_{0}^{\prime}} N_{\alpha\beta} c_{\alpha
+\beta}\omega_{\alpha}\wedge\omega_{\beta}\\
& +\frac{1}{\sqrt{2(n-1)}}\sum_{i<n-1}
\omega_{\epsilon_{n-1}+\epsilon_{n}}\wedge\left(
c_{\epsilon_{i}+\epsilon_{n-1}}\omega_{\epsilon_{i} -\epsilon_{n}}
- c_{\epsilon_{i}+\epsilon_{n}}\omega_{\epsilon_{i}
-\epsilon_{n-1}}\right)\\
& +\frac{1}{\sqrt{2(n-1)}}\sum_{i<n-1}
\omega_{-(\epsilon_{n-1}+\epsilon_{n})}\wedge\left(
c_{\epsilon_{i}-\epsilon_{n}}\omega_{\epsilon_{i} +\epsilon_{n-1}}
- c_{\epsilon_{i}-\epsilon_{n-1}}\omega_{\epsilon_{i}
+\epsilon_{n}}\right)\\
&+\frac{1}{2} \sum_{i<n-1} (\epsilon_{n-1}+\epsilon_{n})\wedge
\left( c_{\epsilon_{i} +\epsilon_{n}}\omega_{\epsilon_{i}
+\epsilon_{n}} + c_{\epsilon_{i}+ \epsilon_{n-1}}\wedge
\omega_{\epsilon_{i}+\epsilon_{n-1}}\right)\\
& -\frac{1}{2} \sum_{i<n-1} (\epsilon_{n-1}+\epsilon_{n})\wedge
\left( c_{\epsilon_{i} -\epsilon_{n}}\omega_{\epsilon_{i}
-\epsilon_{n}} + c_{\epsilon_{i}- \epsilon_{n-1}}\wedge
\omega_{\epsilon_{i}-\epsilon_{n-1}}\right)
\end{align*}
where $\widehat{\omega}_{0}\in \Lambda^{2}(\gh_{0})$ is such that
\[
\widehat{\omega}_{0}(H_{n-1}+H_{n},\cdot )=0,
\]
$a, b,c\in\mathbb{C}$ and $c_{\alpha}\in \mathbb{C}$, for any
$\alpha \in R_{0}^{\prime}.$ The pair $(\gk, \omega)$ is
$\gso_{2n-1,1}$-admissible, hence it defines a regular generalized
complex structure on $SO_{2n-1,1}$,
 if and only if   the real 2-form
$\mathrm{Im}\left(\widehat{\omega}_{0}\right)$ is non-degenerate
on $\gh_{0}\cap\gso_{2n-1,1}$.
\end{thm}

\begin{proof}
To describe  all  closed $2$-forms on $\gk$ ,
we  set
\[
\gp :=\left( \mathrm{Ker}(\epsilon_{n-1}+\epsilon_{n})\cap
\gh_{0}\right) + \gg(R^{\prime}_{0})
\]
and
\[
\gs := \mathbb{C} (H_{n-1}+H_{n}) +
\gg_{\epsilon_{n-1}+\epsilon_{n}} +
\gg_{-(\epsilon_{n-1}+\epsilon_{n})}.
\]
Then   $\gk= \gs \oplus \gp$ is a semidirect decomposition, with
subalgebra $\gs$ and ideal $\gp$,  and we can use Proposition
\ref{rho} to describe all closed $2$-forms on $\gk.$ Decompose
$\omega$ as
\[
\omega = \rho_{0} +\rho_{1}+\rho_{2},
\]
where $\rho_{0}\in\Lambda^{2}(\gs^{*})$ and
$\rho_{1}\in\Lambda^{2}(\gp^{*})$ are the $\gs$, respectively
$\gp$-parts of $\omega$ and $\rho_{2}\in \gp^{*}\wedge \gs^{*}$ is
the mixed part. It is easy to see that any 2-form
\begin{equation}\label{rho-0}
\rho_0 = (\epsilon_{n-1}+\epsilon_{n})\wedge\left(
a\omega_{\epsilon_{n-1}+\epsilon_{n}} +
b\omega_{-(\epsilon_{n-1}+\epsilon_{n})}\right)
+c\omega_{\epsilon_{n-1}+\epsilon_{n}}\wedge
\omega_{-(\epsilon_{n-1}+\epsilon_{n})}
\end{equation}
on $\gs$ is closed. Now we  determine all closed $2$-forms
$\rho_{1}$ on $\gp$. Using (\ref{revizuit}), an argument like in
Theorem \ref{main} shows that any closed $2$-form on $\gp$ is
given by
\begin{equation}\label{condi-2}
\rho_{1}=\widehat{\omega}_{0} +\sum_{\alpha\in
R_{0}^{\prime}}c_{\alpha}\alpha\wedge \omega_{\alpha} +\frac{1}{2}
\sum_{\alpha , \beta\in R_{0}^{\prime}}N_{\alpha\beta} c_{\alpha
+\beta} \omega_{\alpha}\wedge\omega_{\beta}
\end{equation}
where $c_{\alpha}\in {\mathbb C}$ for any $\alpha\in
R_{0}^{\prime}$ and $\widehat{\omega}_{0}$ is any $2$-form defined
on $\mathrm{Ker}(\epsilon_{n-1}+\epsilon_{n})\cap\gh_{0}$. It
remains to determine  the mixed  part $\rho_{2}$ of $\rho$, which
satisfies
\begin{equation}\label{c2-1}
\rho_{2}([s, s^{\prime}], p) + \rho_{2}([s^{\prime}, p], s)
+\rho_{2}([p, s], s^{\prime})=0
\end{equation}
and is related to $\rho_{1}$ by
\begin{equation}\label{c1-1}
\rho_{2}(s, [p, p^{\prime}])  = \rho_{1}([s, p], p^{\prime})
+\rho_{1}(p, [s,p^{\prime}])
\end{equation}
for any $s, s^{\prime}\in \gs$ and $p, p^{\prime}\in \gp .$ A
straightforward computation shows that (\ref{c2-1}) is equivalent
to the
following conditions:\\

i) for any $H\in
\mathrm{Ker}(\epsilon_{n-1}+\epsilon_{n})\cap\gh_{0}$,
\begin{equation}\label{revizuit1}
\rho_{2}(H, \cdot )=0.
\end{equation}

ii) for any $\alpha\in R^{\prime}_{0}$ such that $\alpha
+\epsilon_{n-1}+\epsilon_{n}\notin R$ and $\alpha
-(\epsilon_{n-1}+\epsilon_{n})\notin R$,
$\rho_{2}(H_{n-1}+H_{n}, E_{\alpha})=0$;\\

iii) for any $\alpha\in R_{0}^{\prime}$,
\begin{align*}
\rho_{2}( E_{\epsilon_{n-1}+\epsilon_{n}},
E_{\alpha})=&\frac{N_{\epsilon_{n-1}+\epsilon_{n}, \alpha}}{\alpha
(H_{n-1}+H_{n})+2}\rho_{2} (H_{n-1}+H_{n}, E_{\alpha
+\epsilon_{n-1}+\epsilon_{n}})\\
\rho_{2}( E_{-(\epsilon_{n-1}+\epsilon_{n})},
E_{\alpha})=&\frac{N_{-(\epsilon_{n-1}+\epsilon_{n}),
\alpha}}{\alpha (H_{n-1}+H_{n})-2}\rho_{2} (H_{n-1}+H_{n},
E_{\alpha -(\epsilon_{n-1}+\epsilon_{n})}).
\end{align*}
On the other hand, with $\rho_{1}$ and $\rho_{2}$ as above,
relation (\ref{c1-1}) is equivalent to
\begin{equation}\label{rev1}
c_{\alpha} \alpha (H_{n-1}+H_{n}) =\rho_{2}( H_{n-1}+H_{n},
E_{\alpha}),\quad\forall \alpha\in R_{0}^{\prime}
\end{equation}
(here we use condition (\ref{revizuit})). Combining this relation
with the above expressions of $\rho_{2}(
E_{\epsilon_{n-1}+\epsilon_{n}}, E_{\alpha})$ and $\rho_{2}(
E_{-(\epsilon_{n-1}+\epsilon_{n})}, E_{\alpha})$ we get: for any
$\alpha\in R^{\prime}_{0}$,
\begin{align}
\label{ealpha0}\rho_{2}( E_{\epsilon_{n-1}+\epsilon_{n}},
E_{\alpha})&= N_{\epsilon_{n-1}+\epsilon_{n},\alpha}c_{\alpha
+\epsilon_{n-1}+\epsilon_{n}}\\
\label{ealpha}\rho_{2}( E_{-(\epsilon_{n-1}+\epsilon_{n})},
E_{\alpha})&= N_{-(\epsilon_{n-1}+\epsilon_{n}),\alpha} c_{\alpha
-(\epsilon_{n-1}+\epsilon_{n})}.
\end{align}
From (\ref{revizuit1}), (\ref{rev1}), (\ref{ealpha0}) and
(\ref{ealpha}), we get

\begin{align*}
\rho_{2}&= \sum_{\alpha\in R_{0}^{\prime}}
N_{\epsilon_{n-1}+\epsilon_{n},\alpha}c_{\alpha
+\epsilon_{n-1}+\epsilon_{n}}
\omega_{\epsilon_{n-1}+\epsilon_{n}}\wedge\omega_{\alpha}\\
& + \sum_{\alpha\in R_{0}^{\prime}}
N_{-(\epsilon_{n-1}+\epsilon_{n}),\alpha}c_{\alpha
-(\epsilon_{n-1}+\epsilon_{n})}\omega_{-(\epsilon_{n-1}+\epsilon_{n})}\wedge
\omega_{\alpha} \\
&+\frac{1}{2} \sum_{\alpha\in R^{\prime}_{0}} c_{\alpha} \alpha
(H_{n-1}+ H_{n}) (\epsilon_{n-1}+\epsilon_{n})\wedge
\omega_{\alpha}.
\end{align*}
Using the expression of the structure constants $N_{\alpha\beta}$
we can further simplify $\rho_{2}$ and we conclude:
\begin{align}
\nonumber\rho_{2}&=\frac{1}{\sqrt{2(n-1)}}\sum_{i<n-1}
\omega_{\epsilon_{n-1}+\epsilon_{n}}\wedge\left(
c_{\epsilon_{i}+\epsilon_{n-1}}\omega_{\epsilon_{i} -\epsilon_{n}}
- c_{\epsilon_{i}+\epsilon_{n}}\omega_{\epsilon_{i}
-\epsilon_{n-1}}\right)\\
\nonumber& +\frac{1}{\sqrt{2(n-1)}}\sum_{i<n-1}
\omega_{-(\epsilon_{n-1}+\epsilon_{n})}\wedge\left(
c_{\epsilon_{i}-\epsilon_{n}}\omega_{\epsilon_{i} +\epsilon_{n-1}}
- c_{\epsilon_{i}-\epsilon_{n-1}}\omega_{\epsilon_{i}
+\epsilon_{n}}\right)\\
\nonumber&+\frac{1}{2} \sum_{i<n-1}
(\epsilon_{n-1}+\epsilon_{n})\wedge \left( c_{\epsilon_{i}
+\epsilon_{n}}\omega_{\epsilon_{i} +\epsilon_{n}} +
c_{\epsilon_{i}+ \epsilon_{n-1}}\wedge
\omega_{\epsilon_{i}+\epsilon_{n-1}}\right)\\
\label{fin}& -\frac{1}{2} \sum_{i<n-1}
(\epsilon_{n-1}+\epsilon_{n})\wedge \left( c_{\epsilon_{i}
-\epsilon_{n}}\omega_{\epsilon_{i} -\epsilon_{n}} +
c_{\epsilon_{i}- \epsilon_{n-1}}\wedge
\omega_{\epsilon_{i}-\epsilon_{n-1}}\right) .
\end{align}
Combining (\ref{rho-0}), (\ref{condi-2}) and (\ref{fin}) we get
the first claim. The second claim is trivial.

\end{proof}

\subsection{Generalized complex structures on the real Lie groups
$E_6$ of outer type}\label{ext}

{\bf a) Description of the real forms $(\ge_{6})^{\sigma}$ of outer type of  $\ge_6$}\\

We follow the description of the exceptional  complex Lie algebra
$\ge_6$ given in \cite[p.\,80]{sam}. The  complex Lie algebra
$\ge_{6}$ has dimension $78$ and rank $6$. We take $\gh =
\mathbb{C}^{6}$ for the Cartan subalgebra. Let $\{ e_{1},\cdots ,
e_{6}\}$ be the standard basis of $\mathbb{C}^{6}$, with dual
basis $\{ \epsilon_{1},\cdots , \epsilon_{6}\} .$ The Killing form
restricted to $\gh$ is
\[
\langle x, y\rangle = 24
\sum_{i=1}^{6}\epsilon_{i}(x)\epsilon_{i}(y)
+8(\sum_{i}\epsilon_{i}(x))(\sum_{j}\epsilon_{j}(y)) .
\]
The root system $R$ is formed by $\pm (\epsilon_{i}-\epsilon_{j})$
with $1\leq i < j\leq 6$, $\pm (\epsilon_{i}+\epsilon_{j}+
\epsilon_{k})$ with $1\leq i<j<k\leq 6$ and $\pm
(\epsilon_{1}+\cdots +\epsilon_{6}).$ A system of simple roots is
\[
\Pi = \{ \alpha_{1} = \epsilon_{1}-\epsilon_{2}, \alpha_{2}=
\epsilon_{2}- \epsilon_{3}, \cdots ,
\alpha_{5}=\epsilon_{5}-\epsilon_{6}, \alpha_{6}= \epsilon_{4}
+\epsilon_{5}+\epsilon_{6}\} .
\]
The complex Lie algebra $\ge_{6}$ has two real forms of outer
type, with maximally compact subalgebras $\gf_{4}$ and $\gsp_{4}$.
The Vogan diagram of $\ge_{6}(\gf_{4})$ has no painted node and an
automorphism of order two, given by the horizontal reversing in
the Dynkin diagram of $\ge_{6}.$ The Vogan diagram of
$\ge_{6}(\gsp_{4})$ is the same, the only difference being that
the triple node from the Dynkin diagram is painted, see
\cite[p.\,361]{knapp}. This means that the Cartan involutions of
$\ge_{6}(\gf_{4})$ and $\ge_{6}(\gsp_{4})$ induce the same
canonical order two automorphism of the Dynkin diagram of
$\ge_{6}$, given, in terms of the simple roots from $\Pi$, by
\begin{align*}
\theta (\alpha_{1}) &= \alpha_{5},\quad\theta (\alpha_{2}) = \alpha_{4},\\
\theta (\alpha_{3}) &=\alpha_{3},\quad\theta
(\alpha_{6})=\alpha_{6}.
\end{align*}
Similarly, the defining antiinvolutions $\sigma$ of
$\ge_{6}(\gf_{4})$ and $\ge_{6}(\gsp_{4})$ induce the same action
on $R$:
\[
\sigma (\alpha ) = -\theta (\alpha) , \quad \forall \alpha \in R.
\]
We shall denote by $\gh^{\sigma}=\gh_{\ge_{6}(\gf_{4})}=
\gh_{\ge_{6}(\gsp_{4})}$ the common maximally compact Cartan
subalgebras of $\ge_{6}(\gf_{4})$ and $\ge_{6}(\gsp_{4})$, defined
as in (\ref{cartan-special}). For $\ge_{6}(\gf_{4})$ both roots
$\alpha_{3}$ and $\alpha_{6}$ are compact, while for
$\ge_{6}(\gsp_{4})$, $\alpha_{3}$ is non-compact and $\alpha_{6}$
is compact, with respect to $\gh^{\sigma}$.

\vspace{10pt}

{\bf b) $\sigma$-positive systems of the Lie algebra
$\ge_6$}\\

We denote by the same symbol $\sigma$ the antiinvolutions of
$\ge_{6}$ which define the real forms $\ge_{6}(\gf_{4})$ and
$\ge_{6}(\gsp_{4})$, like in the previous paragraph. The following
Proposition describes all $\sigma$-positive systems in $R$.
\begin{prop}\label{positiveE6}
Any $\sigma$-positive system  of $R$ is equivalent  to one of the
following $\sigma$-positive systems:
\begin{align*}
R_{0}^{(1)}&= \{ \pm (\epsilon_{i} -\epsilon_{j})\}_{1\leq i<j\leq
3}\cup \{ \epsilon_{i}-\epsilon_{j}\}_{i\leq 3,j\geq 4}\cup\{
\epsilon_{i}
+\epsilon_{j}+\epsilon_{k}, \epsilon_{1}+\cdots +\epsilon_{6}\};\\
R_{0}^{(2)}&= \{ \pm (\epsilon_{i} -\epsilon_{j})\}_{1\leq
i<j\leq3} \cup\{ \epsilon_{i}-\epsilon_{j}\}_{i\leq3,j\geq 4}\cup
\{\epsilon_{i}+\epsilon_{j}+\epsilon_{k}\neq\epsilon_{4}+
\epsilon_{5}+\epsilon_{6}\}\\
&\cup \{ -(\epsilon_{4}+\epsilon_{5}+\epsilon_{6}),
\epsilon_{1}+\cdots +\epsilon_{6}\};\\
R_{0}^{(3)}&=\{ \pm (\epsilon_{i}-\epsilon_{j})\}_{1\leq i<j\leq
3} \cup \{ \epsilon_{i}-\epsilon_{j}\}_{i\leq 3,j\geq 4} \cup \{
-(\epsilon_{i}+\epsilon_{j}+\epsilon_{k})\}\\
& \cup\{ -(\epsilon_{1}+\cdots +\epsilon_{6})\};\\
R_{0}^{(4)}&= \{ \pm (\epsilon_{i}-\epsilon_{j})\}_{1\leq i<j\leq
3}\cup \{\epsilon_{i}-\epsilon_{j}\}_{i\leq 3,j\geq 4}\cup \{
-(\epsilon_{i}+
\epsilon_{j}+\epsilon_{k})\neq -(\epsilon_{1}+\epsilon_{2}+\epsilon_{3})\}\\
&\cup \{ \epsilon_{1}+\epsilon_{2}+\epsilon_{3},
-(\epsilon_{1}+\cdots +
\epsilon_{6})\};\\
R_{0}^{(5)}&= \{\pm (\epsilon_{i}-\epsilon_{j})\}_{1\leq i<j\leq
3}\cup\{ \epsilon_{i}-\epsilon_{j}\}_{i\leq 3,j\geq 4}\cup
\{ -(\epsilon_{i}+\epsilon_{5}+\epsilon_{6})\}_{i\leq 4}\\
&\cup \{ -(\epsilon_{i}+\epsilon_{4}+\epsilon_{6})\}_{i\leq 3}
\cup \{\epsilon_{i}+\epsilon_{j}+\epsilon_{6}\}_{i,j\leq 3}\cup \{
- (\epsilon_{i}+\epsilon_{4}+\epsilon_{5})\}_{i\leq 3}\\
& \cup\{ \epsilon_{i}+\epsilon_{j}+\epsilon_{5}\}_{i,j\leq 3} \cup
\{\epsilon_{i}+\epsilon_{j}+\epsilon_{4}\}_{i,j\leq 3} \cup \{
\epsilon_{1}+\epsilon_{2}+\epsilon_{3},\epsilon_{1}+\cdots
+\epsilon_{6}\};\\
R_{0}^{(6)}&=\{\pm (\epsilon_{i}-\epsilon_{j})\}_{1\leq i<j\leq
3}\cup\{ \epsilon_{i}-\epsilon_{j}\}_{i\leq 3,j\geq 4}\cup
\{ -(\epsilon_{i}+\epsilon_{5}+\epsilon_{6})\}_{i\leq 4}\\
&\cup \{ -(\epsilon_{i}+\epsilon_{4}+\epsilon_{6})\}_{i\leq 3}
\cup \{\epsilon_{i}+\epsilon_{j}+\epsilon_{6}\}_{i,j\leq 3}\cup \{
- (\epsilon_{i}+\epsilon_{4}+\epsilon_{5})\}_{i\leq 3}\\
&\cup\{ \epsilon_{i}+\epsilon_{j}+\epsilon_{5}\}_{i,j\leq 3}\cup
\{\epsilon_{i}+\epsilon_{j}+\epsilon_{4}\}_{i,j\leq 3} \cup \{
\epsilon_{1}+\epsilon_{2}+\epsilon_{3},-(\epsilon_{1}+\cdots
+\epsilon_{6})\} .
\end{align*}
Moreover, the image of $R_{0}^{(k)}$ ($1\leq k\leq 6$) through the
antiinvolution $\sigma$ is given by:
\begin{align*}
\sigma (R_{0}^{(1)})&= \{ \pm (\epsilon_{i}-\epsilon_{j})\}_{4\leq
i<j\leq 6}\cup \{ \epsilon_{j}-\epsilon_{i}\}_{i\leq 3,j\geq 4}
\cup\{ -(\epsilon_{i}+\epsilon_{j}+\epsilon_{k}),
-(\epsilon_{1}+\cdots +\epsilon_{6})\};\\
\sigma (R_{0}^{(2)}) &= \{\pm (\epsilon_{i}
-\epsilon_{j})\}_{4\leq i<j\leq 6}\cup \{
\epsilon_{j}-\epsilon_{i}\}_{i\leq 3,j\geq 4}\cup \{
\epsilon_{4}+\epsilon_{5}+\epsilon_{6}, -(\epsilon_{1}+\cdots
+\epsilon_{6})\}\\
&\cup \{-(\epsilon_{i}+\epsilon_{j}+\epsilon_{k})\neq
-(\epsilon_{4}+ \epsilon_{5}+\epsilon_{6})\};\\
\sigma (R_{0}^{(3)})&=\{ \pm (\epsilon_{i}-\epsilon_{j})\}_{4\leq
i<j\leq 6} \cup \{ \epsilon_{j}-\epsilon_{i}\}_{i\leq 3,j\geq 4}
\cup \{ \epsilon_{i}+\epsilon_{j}+\epsilon_{k}\}\cup \{
\epsilon_{1}+\cdots +\epsilon_{6}\} ;\\
\sigma (R_{0}^{(4)})&= \{ \pm (\epsilon_{i}-\epsilon_{j})\}_{4\leq
i<j\leq 6}\cup \{\epsilon_{j}-\epsilon_{i}\}_{i\leq 3,j\geq 4}\cup
\{ \epsilon_{i}+
\epsilon_{j}+\epsilon_{k}\neq \epsilon_{1}+\epsilon_{2}+\epsilon_{3}\}\\
&\cup \{ -(\epsilon_{1}+\epsilon_{2}+\epsilon_{3}),
\epsilon_{1}+\cdots + \epsilon_{6}\} ;\\
\sigma (R_{0}^{(5)}) &= \{\pm (\epsilon_{i}-\epsilon_{j})\}_{4\leq
i<j\leq 6}\cup\{ \epsilon_{j}-\epsilon_{i}\}_{i\leq 3,j\geq 4}\cup
\{ \epsilon_{i}+\epsilon_{5}+\epsilon_{6}\}_{i\leq 4}\\
&\cup \{ \epsilon_{i}+\epsilon_{4}+\epsilon_{6}\}_{i\leq 3} \cup
\{-( \epsilon_{i}+\epsilon_{j}+\epsilon_{6}) \}_{i,j\leq 3}\cup \{
\epsilon_{i}+\epsilon_{4}+\epsilon_{5}\}_{i\leq 3}\\
&\cup\{ -(\epsilon_{i}+\epsilon_{j}+\epsilon_{5})\}_{i,j\leq
3}\cup \{-(\epsilon_{i}+\epsilon_{j}+\epsilon_{4})\}_{i,j\leq 3}
\cup \{ -(\epsilon_{1}+\epsilon_{2}+\epsilon_{3})\}\\
&\cup \{ -(\epsilon_{1}+\cdots +\epsilon_{6})\} ;\\
\sigma (R_{0}^{(6)}) &= \{\pm (\epsilon_{i}-\epsilon_{j})\}_{4\leq
i<j\leq 6}\cup\{ \epsilon_{j}-\epsilon_{i}\}_{i\leq 3,j\geq 4}
\cup\{ \epsilon_{i}+\epsilon_{5}+\epsilon_{6}\}_{i\leq 4}\\
&\cup \{ \epsilon_{i}+\epsilon_{4}+\epsilon_{6}\}_{i\leq 3} \cup
\{-( \epsilon_{i}+\epsilon_{j}+\epsilon_{6}) \}_{i,j\leq 3}\cup \{
\epsilon_{i}+\epsilon_{4}+\epsilon_{5}\}_{i\leq 3}\\
&\cup\{ -(\epsilon_{i}+\epsilon_{j}+\epsilon_{5})\}_{i,j\leq 3}
\cup \{-(\epsilon_{i}+\epsilon_{j}+\epsilon_{4})\}_{i,j\leq 3}
\cup \{ -(\epsilon_{1}+\epsilon_{2}+\epsilon_{3})\}\\
&\cup \{\epsilon_{1}+\cdots +\epsilon_{6}\} .
\end{align*}
\end{prop}

\begin{proof}
The proof is long but straightforward, and uses the action of
$\sigma$ on roots and the properties of $\sigma$-positive systems
(see Definition \ref{sigmapara}). We give only the sketch of the
argument. Let $R_{0}$ be a $\sigma$-positive system of $R$ and
\[
\tilde{R}= [\alpha_{1},\alpha_{2}, \alpha_{4},\alpha_{5}]
\]
the closed set of roots consisting of all roots from $R$ which are
linear combinations of  the simple roots $\alpha_{i}$, for
$i=1,2,4,5.$ Since $\tilde{R}$ has two irreducible components and
$\tilde{R}=(\tilde{R}\cap R_{0})\cup\left( \tilde{R}\cap\sigma
(R_{0})\right)$ is a decomposition of $\tilde{R}$ into a disjoint
union of two closed subsets, we may assume, replacing if
necessary, $R_{0}$ by $\sigma (R_{0})$, that:
\[
\tilde{R} \cap R_{0}= [\alpha_{1}, \alpha_{2}],\quad \tilde{R}\cap
\sigma (R_{0}) = [\alpha_{4}, \alpha_{5}].
\]
In particular,
\[
[\alpha_{1}, \alpha_{2}]= \{\pm (\epsilon_{1}-\epsilon_{2}), \pm
(\epsilon_{2}-\epsilon_{3}), \pm
(\epsilon_{1}-\epsilon_{3})\}\subset R_{0}.
\]
Since $\sigma$ interchanges the pairs $(\alpha_{3}, \alpha_{6})$
and $(-\alpha_{3}, -\alpha_{6})$, replacing again, if necessary,
$R_{0}$ by $-R_{0}$, we have the following two possibilities:
either
\begin{equation}\label{A}
\{ \pm (\epsilon_{1}-\epsilon_{2}), \pm
(\epsilon_{2}-\epsilon_{3}), \pm (\epsilon_{1}-\epsilon_{3}),
\alpha_{3}=\epsilon_{3}-\epsilon_{4},\alpha_{6}= \epsilon_{4}+\epsilon_{5}+
\epsilon_{6}\} \subset R_{0}
\end{equation}
or
\begin{equation}\label{AA}
\{ \pm (\epsilon_{1}-\epsilon_{2}), \pm
(\epsilon_{2}-\epsilon_{3}), \pm (\epsilon_{1}-\epsilon_{3}),
\alpha_{3}=\epsilon_{3}-\epsilon_{4},-\alpha_{6}= -(\epsilon_{4}+\epsilon_{5}+
\epsilon_{6}) \} \subset R_{0}.
\end{equation}
Using that $R_{0}$ is a $\sigma$-positive system one may check if
(\ref{A}) holds then $R_{0} = R_{0}^{(1)}$. Assume now that
(\ref{AA}) holds. We distinguish two further subcases: when
\begin{equation}\label{sub1}
\epsilon_{3}+\epsilon_{5}+\epsilon_{6}\in R_{0}
\end{equation}
and, respectively, when
\begin{equation}\label{sub2}
\epsilon_{3}+\epsilon_{5}+\epsilon_{6}\in\sigma ( R_{0}).
\end{equation}
When (\ref{AA}) and (\ref{sub1}) hold, $R_{0}= R_{0}^{(2)}.$
Assume now that (\ref{AA}) and (\ref{sub2}) hold. It turns out
that if
\begin{equation}\label{sta}
-(\epsilon_{2}+ \epsilon_{3}+\epsilon_{6})\in R_{0}
\end{equation}
then $R_{0}= R_{0}^{(3)}$, provided that
$-(\epsilon_{1}+\epsilon_{2}+\epsilon_{3})\in R_{0}$, or $R_{0}=
R_{0}^{(4)}$, provided that
$-(\epsilon_{1}+\epsilon_{2}+\epsilon_{3})\in\sigma ( R_{0})$. If
(\ref{AA}) and (\ref{sub2}) hold but (\ref{sta}) does not, then
$R_{0}= R_{0}^{(5)}$, provided that $\epsilon_{1}+\cdots
+\epsilon_{6}\in R_{0}$, or $R_{0}= R_{0}^{(6)}$, provided that
$\epsilon_{1}+\cdots +\epsilon_{6}\in\sigma ( R_{0})$.

\end{proof}

\vspace{10pt}

{\bf c) $\ge_6^{\sigma}$-admissible  pairs}\\

We preserve the notations from the previous paragraphs. We denote
by
\[
\gk^{(k)} := \gh_{0}^{(k)} + \gg( R_{0}^{(k)})
\]
a regular subalgebra of $\ge_{6}$, normalized by $\gh^{\sigma}$,
with root system $R_{0}^{(k)}$ ($1\leq k\leq 6$) described in
Proposition \ref{positiveE6} and with Cartan part
$\gh_{0}^{(k)}\subset\gh$, satisfying $\gh_{0}^{(k)} +\sigma
(\gh_{0}^{(k)}) = \gh$. Since
\[
R_{0}^{(k), \mathrm{sym}}= \{ \pm (\epsilon_{i}-\epsilon_{j}),
1\leq i\neq j\leq 3\}
\]
and
\[
\sigma (R_{0}^{(k), \mathrm{sym}})= \{ \pm
(\epsilon_{i}-\epsilon_{j}), 4\leq i\neq j\leq 6\}
\]
for any $1\leq k\leq 6$, condition (\ref{condfin-1}) is satisfied.
Moreover,
\[
\mathrm{Span} (R_{0}^{(k), \mathrm{sym}})^{\flat} = \{
\sum_{i=1}^{3}\lambda_{i}e_{i},\quad \sum_{i=1}^{3}\lambda_{i}=0\}
\]
is included in $\gh_{0}^{(k)}$, because $\gk^{(k)}$ is a
subalgebra.
\begin{thm} \label{e66} For any $1\leq k\leq 6$, the $2$-form $\omega_{(k)}$
on  $\gk^{(k)}$,  defined by
\begin{align*}
\omega_{(1)}& :=\widehat{\omega}_{(1)}+\lambda_{1} (\epsilon_{4}
+\epsilon_{5}+\epsilon_{6})\wedge
\omega_{\epsilon_{4}+\epsilon_{5}+\epsilon_{6}}\\
\omega_{(2)}&:= \widehat{\omega}_{(2)}+\lambda_{2} (\epsilon_{4}
+\epsilon_{5}+\epsilon_{6})\wedge
\omega_{-(\epsilon_{4}+\epsilon_{5}+\epsilon_{6})}\\
\omega_{(3)}&:= \widehat{\omega}_{(3)}+\lambda_{3}(\epsilon_{1}
+\epsilon_{2}+\epsilon_{3})\wedge
\omega_{-(\epsilon_{1}+\epsilon_{2}+\epsilon_{3})}\\
\omega_{(4)}&:= \widehat{\omega}_{(4)}+ \lambda_{4}(\epsilon_{1}
+\epsilon_{2}+\epsilon_{3})\wedge
\omega_{\epsilon_{1}+\epsilon_{2}+\epsilon_{3}}\\
\omega_{(5)}&:= \widehat{\omega}_{(5)}+\lambda_{5}(\epsilon_{1}
+\cdots +\epsilon_{6})\wedge
\omega_{\epsilon_{1}+\cdots +\epsilon_{6}}\\
\omega_{(6)}&:= \widehat{\omega}_{(6)}+ \lambda_{6}(\epsilon_{1}
+\cdots +\epsilon_{6})\wedge \omega_{-(\epsilon_{1}+\cdots
+\epsilon_{6})}
\end{align*}
where $\lambda_{k}\in\mathbb{C}$ and $\widehat{\omega}_{(k)}$ is a
$2$-form on $\gh_{0}^{(k)}$ with
\begin{equation}\label{conditie-o}
\widehat{\omega}_{(k)}(e_{i}-e_{j},\cdot )=0,\quad\forall 1\leq
i,j\leq 3 ,
\end{equation}
is  closed. If, moreover,
$\mathrm{Im}\left(\widehat{\omega}_{(k)}\right)$ is non-degenerate
on $\gh_{0}^{(k)}\cap \gh^{\sigma}$, then
$(\gk^{(k)},\omega_{(k)})$ is an admissible pair and it defines a
regular generalized complex structure on the real Lie group
$(E_{6})^{\sigma}.$
\end{thm}
\begin{proof}
We only prove the statement for $\omega_{(1)}$, the proofs for the
remaining $\omega_{(k)}$, $2\leq k\leq 6$, being similar. Recall
the formula for the exterior derivative of a covector $\beta \in
(\gk^{(1)})^{*}$:
\begin{equation}\label{e6-1}
d_{\gk^{(1)}}\beta (X, Y) = -\beta ([X,Y]), \quad\forall X, Y\in
\gk^{(1)}.
\end{equation}
From (\ref{conditie-o}) and (\ref{e6-1}), $\widehat{\omega}_{(1)}$
is closed on $\gk^{(1)}$ (see also (\ref{conditie-closed})). Since
the root $\epsilon_{4} +\epsilon_{5} +\epsilon_{6}\in R_{0}^{(1)}$
cannot be written as a sum of two roots, both from $R_{0}^{(1)}$,
relation (\ref{e6-1}) implies that
\begin{equation}\label{e61}
d_{\gk^{(1)}}\left(
\omega_{\epsilon_{4}+\epsilon_{5}+\epsilon_{6}}\right)=
-(\epsilon_{4}+\epsilon_{5}+\epsilon_{6})\wedge
\omega_{\epsilon_{4}+ \epsilon_{5}+\epsilon_{6}}.
\end{equation}
Also,
\begin{equation}\label{e6-2}
d_{\gk^{(1)}}(\epsilon_{4}+\epsilon_{5}+\epsilon_{6})=0,
\end{equation}
where in (\ref{e6-2}) we used that
$\epsilon_{4}+\epsilon_{5}+\epsilon_{6}$ annihilates
$\mathrm{Span} (R_{0}^{(1), \mathrm{sym}})^{\flat}.$ From
(\ref{e61}) and (\ref{e6-2}),
$(\epsilon_{4}+\epsilon_{5}+\epsilon_{6})\wedge
\omega_{\epsilon_{4}+\epsilon_{5}+\epsilon_{6}}$ is closed on
$\gk^{(1)}.$ Thus $\omega_{(1)}$ is also closed.

\end{proof}

\begin{rem}\begin{enumerate}
\item  All  regular $\gg$-admissible pairs $(\gk ,\omega )$ from this section have the property that
the root system $R_{0}$ of  $\gk$ is a $\sigma$-positive system.
The problem of constructing more general regular $\gg$-admissible pairs,
with $\sigma$-parabolic system $R_{0}$, is left open.

\item
In this paper we studied in detail regular generalized complex structures
on semisimple Lie groups. The problem of constructing invariant generalized complex structures on these groups,
which are not regular, remains open. Further investigation in this direction is needed and will hopefully be done in a forthcoming paper.

\end{enumerate}

\end{rem}

\vspace{10pt}

{\bf AUTHORS ADDRESS:}\\

D.V.ALEKSEEVSKY:  Edinburgh University, The King's Buildings,
JCMB, Mayfield Road, Edinburgh, EH9 3JZ, Great Britain and Hamburg
University, Bundesstrasse 55, Hamburg 20146, Germany; e-mail:
D.Aleksee@ed.ac.uk;\\

L. DAVID: Institute of Mathematics Simion Stoilow of the Romanian
Academy, Calea Grivitei no. 21, Sector 1, Bucharest, Romania;
liana.david@imar.ro


\begin{thebibliography}{99}

\bibitem{ref1} L. C. de Andres, M. L. Barberis, I. Dotti and M.
Fernandez {\em Hermitian structures on cotangent bundles of four
dimensional solvable Lie groups}, Osaka J. Math. 44 no. 4 (2007),
765-793.

\bibitem{dmitri} D. V. Alekseevsky and A. M. Perelomov:
{\em Invariant K\"{a}hler-Einstein metrics on compact homogeneous
spaces}, Funct. Anal. Applic. 20 (1986), no. 3, 171-182.

\bibitem{sdmitri} D. V. Alekseevsky and A. M. Perelomov: {\em Poisson and symplectic
 structures on Lie algebras, I},  J. Geom. Physics, 22  (1997), 191-211.

\bibitem{bourbaki} N. Bourbaki: {\em Lie groups and Lie
algebras, Chapitres 4-6} (Springer 2002).


\bibitem{burstall} F. Burstall and J. Rawnsley: {\em Twistor
Theory for Symmetric Spaces}, vol. 1424, (Springer 1990).

\bibitem{cavalcanti}  G. R. Cavalcanti and M. Gualtieri:
{\em Generalized complex structures on nilmanifolds}, J.
Symplectic Geom., 2 (2004), 393-410.

\bibitem{ref2}  A. Fino and A. Tomassini: {\em Non-K\"{a}hler
Solvmanifolds with Generalized K\"{a}hler structure},  J.
Symplectic Geom. 7 no. 2 (2009), 1-14.

\bibitem{thesis}  M. Gualtieri: {\em Generalized Complex Geometry},
Ph.D thesis, University of Oxford, 2003; arxiv:mathDG/0401221.


\bibitem{hitchin} N. G. Hitchin: {\em Generalized Calabi-Yau manifolds},
Q. J. Math. 54 no. 3 (2003), 281-308.

\bibitem{knapp}  A. W. Knapp: {\em Lie Groups Beyond an
Introduction}, Progress in Mathematics, vol. 140, (Birkhauser
1996).


\bibitem{brett}  B. Milburn: {\em Generalized Complex and
Dirac Structures on Homogeneous Spaces}, arxiv:0712.2627v2.


\bibitem{oni} A. L. Onishchik and E. B. Vinberg: {\em Lie
groups and Lie algebras III}, (Springer-Verlag Berlin Heidelberg
New-York London Paris Tokyo Hong-Kong Barcelona Budapest).

\bibitem{ornea}  L. Ornea and R. Pantilie: {\em Holomorphic maps
between generalized complex manifolds},  J. Geom. Physics, 61
(2011), 1502-1515.


\bibitem{sam} H. Samelson: {\em Notes on Lie algebras}, (Springer-Verlag, 1989).

\bibitem{wang} H. C. Wang: {\em Closed manifolds with homogeneous
complex structures},  Amer. J. Math. 76 (1954), 1-32.

\end{thebibliography}
\end{document}